\documentclass[11pt]{article}

\usepackage{amssymb,amsmath,amsthm,epsfig,mathtools,amscd,tikz,amsfonts}

\theoremstyle{plain}
\newtheorem{thm}{Theorem}[section]
\newtheorem{lem}{Lemma}[section]

\theoremstyle{definition}
\newtheorem{rem}{Remark}[section]

\numberwithin{equation}{section}
\numberwithin{defn}{section}
\numberwithin{figure}{section}

\begin{document}

\title{Pointwise Bounds and Blow-up for Systems of Nonlinear Fractional Parabolic Inequalities}

\author{Steven D. Taliaferro\\
Department of Mathematics\\
Texas A\&M University\\
College Station, TX 77843}

\date{}
\maketitle	

\begin{abstract}
We investigate nonnegative
solutions $u(x,t)$ and $v(x,t)$ of the nonlinear system of
inequalities
\[
 \begin{rcases}
  0\leq(\partial_t -\Delta)^\alpha u\leq v^\lambda \\
  0\leq (\partial_t -\Delta)^\beta v\leq u^\sigma
 \end{rcases}
 \quad\text{ in }\mathbb{R}^n \times\mathbb{R},\,n\geq 1
\]
satisfying the initial conditions
\[
 u=v=0\quad\text{ in }\mathbb{R}^n \times(-\infty,0)
\]
where $\lambda,\sigma,\alpha$, and $\beta$ are positive constants.

Specifically, using the definition of the fractional heat operator
$(\partial_t-\Delta)^\alpha$ given in \cite{T}, we obtain, when they
exist, optimal pointwise upper bounds on
$\mathbb{R}^n \times(0,\infty)$ for nonnegative solutions $u$ and $v$
of this initial value problem with particular
emphasis on these bounds as $t\to0^+$ and as $t\to\infty$.  
\medskip


\end{abstract}

\section{Introduction}\label{sec1}
In this paper we study pointwise upper bounds for nonnegative
solutions $u(x,t)$ and $v(x,t)$ of the nonlinear system of
inequalities
\begin{equation}\label{I.1}
 \begin{rcases}
  0\leq(\partial_t -\Delta)^\alpha u\leq v^\lambda \\
  0\leq (\partial_t -\Delta)^\beta v\leq u^\sigma
 \end{rcases}
 \quad\text{ in }\mathbb{R}^n \times\mathbb{R},\,n\geq 1
\end{equation}
satisfying the initial conditions
\begin{equation}\label{I.2}
 u=v=0\quad\text{ in }\mathbb{R}^n \times(-\infty,0),
\end{equation}
where $\lambda,\sigma,\alpha$, and $\beta$ are positive constants.

Our results in this paper for the system \eqref{I.1}, \eqref{I.2} are
an extension of our results in \cite{T} on pointwise bounds for
nonnegative solutions $u(x,t)$ of the scalar initial value problem
\[  
0\leq(\partial_t -\Delta)^\alpha u\leq u^\lambda
\quad\text{ in }\mathbb{R}^n \times\mathbb{R}
\]
\[
 u=0\quad\text{ in }\mathbb{R}^n \times(-\infty,0),
\]
where $\lambda$ and $\alpha$ are positive constants.

As in \cite{T}, we define the fully fractional heat operator
\begin{equation}\label{I.3}
 (\partial_t -\Delta)^\alpha :Y^{p}_{\alpha}\to X^p
\end{equation}
for
\begin{equation}\label{I.4}
 \left(p>1\text{ and }0<\alpha<\frac{n+2}{2p}\right)\quad\text{ or }\quad\left(p=1\text{ and }0<\alpha\leq\frac{n+2}{2p}\right)
\end{equation}as the inverse of the operator
\begin{equation}\label{I.5}
 J_\alpha :X^p \to Y^{p}_{\alpha}
\end{equation}
where
\begin{equation}\label{I.6}
 X^p:=\bigcap_{T\in\mathbb{R}}L^p (\mathbb{R}^n \times\mathbb{R}_T ),\quad\mathbb{R}_T :=(-\infty,T),
\end{equation}
\begin{equation}\label{I.7}
 J_\alpha f(x,t):=\iint_{\mathbb{R}^n \times\mathbb{R}_t}\Phi_\alpha (x-\xi,t-\tau)f(\xi,\tau)\,d\xi \,d\tau
\end{equation}
and
\begin{equation}\label{I.8}
 Y^{p}_{\alpha}:=J_\alpha (X^p ).
\end{equation}
By \eqref{I.6} we mean $X^p$ is the set of all measurable functions
$f:\mathbb{R}^n \times\mathbb{R}\to\mathbb{R}$ such that
$$\| f\|_{L^p (\mathbb{R}^n \times\mathbb{R}_T )}<\infty\quad\text{ for all }T\in\mathbb{R}.$$
In the definition \eqref{I.7} of $J_\alpha$,
\begin{equation}\label{I.9}
\Phi_\alpha (x,t):=\frac{t^{\alpha-1}}{\Gamma(\alpha)}\,\frac{1}{(4\pi
  t)^{n/2}}e^{-|x|^2 /(4t)}\raisebox{2pt}{$\chi$}_{(0,\infty)}(t)
\end{equation}
is the fractional heat kernel.

When $p$ and $\alpha$ satisfy \eqref{I.4}, it was shown in \cite{T}
that the operator \eqref{I.5} has among others the following properties:
\begin{enumerate}
\item[(P1)] it makes sense because
  $J_\alpha f\in L^{p}_{\text{loc}}(\mathbb{R}^n
  \times\mathbb{R})\text{ for }f\in X^p$,
\item[(P2)] it is one-to-one and onto,
\item[(P3)] if $f\in X^p$ and $u=J_\alpha f$ then $f=0$ in
  $\mathbb{R}^n \times(-\infty,0)$ if and only if $u=0$ in
  $\mathbb{R}^n \times(-\infty,0)$.
\end{enumerate}
By properties (P1) and (P2) we can indeed define \eqref{I.3} as the
inverse of \eqref{I.5} when $p$ and $\alpha$ satisfy \eqref{I.4}.
Property (P3) will be needed to handle the initial conditions
\eqref{I.2}.

According to our results in Section \ref{sec2} there are essentially
only three possibilities for nonnegative solutions $u\in Y^p_\alpha$
and $v\in Y^q_\beta $ of \eqref{I.1}, \eqref{I.2} depending on $n$,
$\lambda$, $\sigma$, $\alpha$, $\beta$, $p$, and $q$:

\begin{enumerate}
\item[(i)] The only solution is $u\equiv v \equiv 0$ in
  $\mathbb{R}^n\times\mathbb{R}$;
\item[(ii)] There exist sharp nonzero pointwise bounds for solutions as
  $t\to 0^+$ and as $t\to\infty$;
\item[(iii)] There do not exist pointwise bounds for solutions as
  $t\to 0^+$ and as $t\to\infty$.
\end{enumerate}
All possiblities can occur. For the precise statements of
possibilities (i), (ii), and (iii) see Theorem \ref{thm2.1}, Theorems
\ref{thm2.2}--\ref{thm2.4}, and Theorems \ref{thm2.5} and
\ref{thm2.6}, respectively.

The operator  \eqref{I.3} is a fully fractional heat operator as
opposed to time fractional heat operators in which the fractional
derivatives are only with respect to $t$, and space fractional heat
operators, in which the fractional derivatives are only with respect to $x$.

Some recent results for nonlinear PDEs containing time (resp. space)
fractional heat operators can be found in \cite{AV,A,ACV,DVV,K,
  KSVZ,M,OD,SS,VZ,ZS}
(resp. \cite{AABP,AMPP,BV,CVW,DS,FKRT,GW,JS,MT,PV,V,VV,VPQR}). Except
for \cite{T}, we know of no results for nonlinear PDEs containing the
fully fractional heat operator $(\partial_t -\Delta)^\alpha$. However
results for linear PDEs containing this operator, including in
particular
\[
(\partial_t-\Delta)^\alpha u=f,
\]
where $f$ is a given function, can be found in \cite{ACM,NS,SK,ST}.

\section{Statement of Results}\label{sec2}
In this section we state our results concerning pointwise bounds for
nonnegative solutions
\begin{equation}\label{2.1}
 u\in Y^{p}_{\alpha}\quad\text{ and }\quad v\in Y^{q}_{\beta}
\end{equation}
of the nonlinear system of inequalities
\begin{equation}\label{2.2}
 \begin{rcases}
  0\leq(\partial_t -\Delta)^\alpha u\leq v^\lambda \\
  0\leq(\partial_t -\Delta)^\beta v\le u^\sigma
 \end{rcases}
 \quad\text{ in }\mathbb{R}^n \times\mathbb{R},\,n\geq1
\end{equation}
satisfying the initial conditions
\begin{equation}\label{2.3}
 u=v=0\quad\text{ in }\mathbb{R}^n \times(-\infty,0),
\end{equation}
where
\begin{equation}\label{2.4}
 p,q\in[1,\infty),\quad\lambda,\sigma,\alpha,\beta\in(0,\infty),
\end{equation}
and, as in the definition in Section \ref{sec1} of the operator \eqref{I.3}, $p$ and $\alpha$ satisfy 
\begin{equation}\label{2.5}
 \left(p>1\text{ and }0<\alpha<\frac{n+2}{2p}\right)\quad\text{ or }\quad\left(p=1\text{ and }0<\alpha\leq\frac{n+2}{2p}\right)
\end{equation}
and $q$ and $\beta$ satisfy
\begin{equation}\label{2.6}
 \left(q>1\text{ and }0<\beta<\frac{n+2}{2q}\right)\quad\text{ or }\quad\left(q=1\text{ and }0<\beta\leq\frac{n+2}{2q}\right).
\end{equation}

If $p$ and $\alpha$ satisfy \eqref{2.5}, $u\in Y^{p}_{\alpha}$, and
$(\partial_t -\Delta)^\alpha u\geq0$ in
$\mathbb{R}^n \times\mathbb{R}$ then
\[u=J_\alpha ((\partial_t -\Delta)^\alpha u)\geq0\quad\text{ in
}\mathbb{R}^n \times\mathbb{R}\]
by \eqref{I.7} and the nonnegativity of $\Phi_\alpha$.  Thus the
assumption that $u$ (and similarly $v$) is nonnegative can be omitted
when studying the problem \eqref{2.1}--\eqref{2.6}.

Moreover, when studying the problem \eqref{2.1}--\eqref{2.6}, we can
assume without loss of generality that
\begin{equation}\label{2.7}
 (n+2-2p\alpha)q^2 \sigma\leq(n+2-2q\beta)p^2 \lambda,
\end{equation}
for otherwise switch the symbols for $u,\lambda,\alpha$, and $p$ with
the symbols for $v,\sigma,\beta$, and $q$ respectively.

If \eqref{2.5} holds then either
\begin{equation}\label{2.8}
 2p\alpha<n+2
\end{equation}
or
\begin{equation}\label{2.9}
 2p\alpha=n+2.
\end{equation}
The following Theorems \ref{thm2.1}--\ref{thm2.6} deal with solutions
of \eqref{2.1}--\eqref{2.3} when \eqref{2.4}--\eqref{2.7} and
\eqref{2.8} hold; the only exception being that \eqref{2.7} and
\eqref{2.8} are not assumed in Theorems \ref{thm2.3} and \ref{thm2.4}.
Theorem \ref{thm2.7} deals with solutions of \eqref{2.1}--\eqref{2.3} in the
simpler case when \eqref{2.4}--\eqref{2.7} and \eqref{2.9} hold.

If \eqref{2.4}, \eqref{2.7}, and \eqref{2.8} hold then
$$2q\beta<n+2,$$
\begin{equation}\label{2.10}
 0<\sigma\leq\nu(\lambda):=\frac{(n+2-2q\beta)p^2}{(n+2-2p\alpha)q^2}\lambda
\end{equation}
and the curves $\sigma=\nu(\lambda)$ and
\begin{equation}\label{2.11}
 \sigma=\mu(\lambda):=\frac{2p\beta}{n+2-2p\alpha}+\frac{n+2}{(n+2-2p\alpha)\lambda}
\end{equation}
intersect at
\begin{equation}\label{2.12}
 (\lambda_0 ,\sigma_0 )=\left(\frac{(n+2)q}{(n+2-2q\beta)p},\,\frac{(n+2)p}{(n+2-2p\alpha)q}\right).
\end{equation}
See Figure \ref{fig1}.  Thus assuming \eqref{2.4}, \eqref{2.7}, and
\eqref{2.8} hold, the point $(\lambda,\sigma)$ belongs to one of the
following five pairwise disjoint subsets of the $\lambda\sigma$-plane.
\begin{align*}
 &A:=\{(\lambda,\sigma):0<\sigma\leq\nu(\lambda)\text{ and }1/\lambda\leq\sigma<\mu(\lambda)\}\\
 &B:=\{(\lambda,\sigma):0<\sigma\leq\nu(\lambda)\text{ and }\sigma<1/\lambda\}\\
 &C:=\{(\lambda,\sigma):\mu(\lambda)<\sigma\leq\sigma_0 \text{ and }\lambda>0\}\\
 &D:=\{(\lambda,\sigma):\sigma_0 <\sigma\leq\nu(\lambda)\}\\
 &E:=\{(\lambda,\sigma):\sigma=\mu(\lambda)\text{ and }\lambda\geq\lambda_0 \}.
\end{align*}

Note that $A,B,C$, and $D$ are two dimensional regions in the
$\lambda\sigma$-plane whereas $E$ is the curve separating $A$ and $C$.
(See Figure \ref{fig1}.)

\begin{figure}
\begin{tikzpicture}[xscale=1.5]

\draw [fill=lime, lime] (1.7,1.7) to [out=-50, in=178] (7,0.1) -- (7,0)
-- (0,0);
\draw [fill=orange, orange] (3,3) to [out=-50, in=178] (7,1.1) -- (7,0.10) to
[out=178, in=-50] (1.7,1.7); 
\draw [fill=yellow, yellow] (3,3) -- (7,3) -- (7,1.1) to [out=178,
in=-50] (3,3);
\draw [fill=cyan, cyan] (3,3) -- (7,3) -- (7,7);

\draw [<->] [thick] (0,7.5) -- (0,0) -- (7.5,0);
\draw [thick] (0,0) -- (7,7);
\draw [thick] (1.7,1.7) to [out=-50, in=178] (7,0.10);
\draw [thick] (3,3) to [out=-50, in=178] (7,1.1);
\draw [thick] (3,3) -- (7,3);
\draw [dashed] (0,1) -- (7,1);
\draw [dashed] (0,3) -- (3,3);
\draw [dashed] (3,0) -- (3,3);
\node [right] at (7.5,0) {$\lambda$};
\node [above] at (0,7.5) {$\sigma$};
\node [left] at (0,3) {$\hskip 0.7in \sigma_0$};
\node [below] at (3.05,0) {$\lambda_0$};
\node [right] at (1.4,0.5) {\huge $B$};
\node [right] at (3.2,1.5) {\huge $A$};
\node [right] at (5.0,2.2) {\huge $C$};
\node [right] at (5.0,4.0) {\huge $D$};
\node [right] at (5.6,0.35) {$\sigma=1/\lambda$};
\node [right] at (5.6,1.45) {$\sigma=\mu(\lambda$)};
\node [right] at (3.7,5.0) {$\sigma=\nu(\lambda$)};

\end{tikzpicture}
\caption{Graphs of regions A, B, C, and D.}
\label{fig1}
\end{figure}
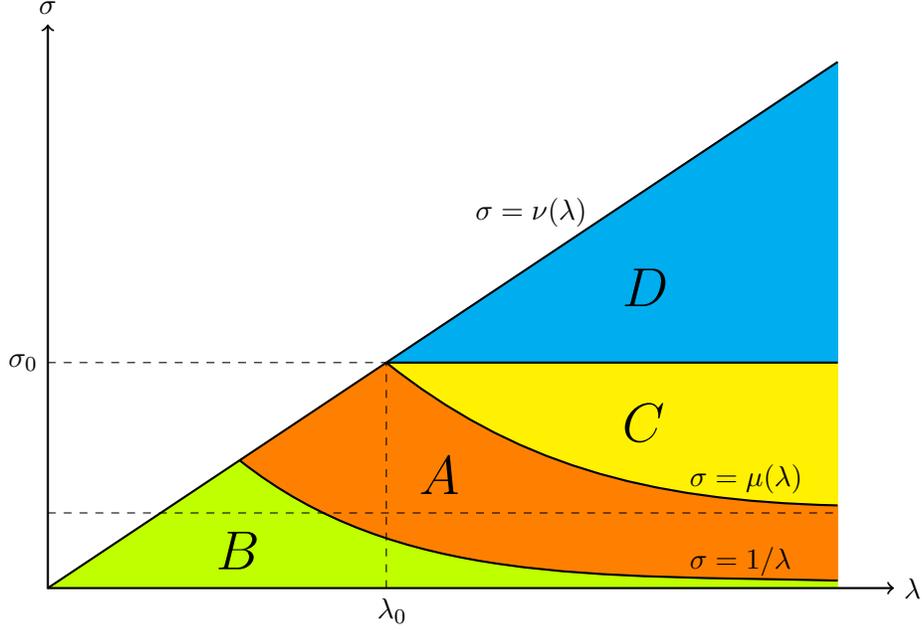

Theorems \ref{thm2.1}--\ref{thm2.6} deal with solutions of
\eqref{2.1}--\eqref{2.3} when \eqref{2.4}--\eqref{2.7} and \eqref{2.8}
hold and $(\lambda,\sigma)$ is in $A,B,C$ or $D$.  We have no results
when $(\lambda,\sigma)\in E$.

The following theorem deals with solutions of \eqref{2.1}--\eqref{2.3}
when \eqref{2.4}--\eqref{2.7} and \eqref{2.8} hold and
$(\lambda,\sigma)\in A$.

\begin{thm}\label{thm2.1}
	Suppose \eqref{2.1}--\eqref{2.7} and \eqref{2.8} hold and
	$$\frac{1}{\lambda}\leq\sigma<\mu(\lambda).$$
	Then
	$$u=v=0\quad\text{ a.e. in }\mathbb{R}^n \times\mathbb{R}.$$
\end{thm}

The following theorem deals with solutions of \eqref{2.1}--\eqref{2.3}
when \eqref{2.4}--\eqref{2.7} and \eqref{2.8} hold and
$(\lambda,\sigma)\in B$.

\begin{thm}\label{thm2.2}
	Suppose \eqref{2.1}--\eqref{2.7} and \eqref{2.8} hold and
	\begin{equation}\label{2.13}
	 \sigma<1/\lambda.
	\end{equation}
	Then for all $T>0$ we have
	\begin{equation}\label{2.14}
	 \| u\|_{L^\infty (\mathbb{R}^n \times(0,T))}\leq M^{\frac{\lambda}{1-\lambda\sigma}}_{2}T^{\gamma_2 /\sigma}
	\end{equation}
	and
	\begin{equation}\label{2.15}
	 \| v\|_{L^\infty (\mathbb{R}^n \times(0,T))}\leq M^{\frac{\sigma}{1-\lambda\sigma}}_{1}T^{\gamma_1 /\lambda}
	\end{equation}
	where
	\begin{equation}\label{2.16}
          \gamma_1
          =\frac{(\beta+\alpha\sigma)\lambda}{1-\lambda\sigma},
          \quad\gamma_2 =\frac{(\alpha+\beta\lambda)\sigma}{1-\lambda\sigma},
	\end{equation}
	\begin{equation}\label{2.17}
	 M_1 =\frac{\Gamma(\gamma_1 +1)\Gamma(\gamma_2 +1)^{1/\sigma}}{\Gamma(\alpha+\gamma_1 +1)\Gamma(\beta+\gamma_2 +1)^{1/\sigma}},
	\end{equation}
	\begin{equation}\label{2.18}
	 M_2 =\frac{\Gamma(\gamma_2 +1)\Gamma(\gamma_1 +1)^{1/\lambda}}{\Gamma(\beta+\gamma_2 +1)\Gamma(\alpha+\gamma_1 +1)^{1/\lambda}},
	\end{equation}
	where $\Gamma$ is the Gamma function. 
\end{thm}
	
By the following theorem, the bounds \eqref{2.14} and \eqref{2.15} in
Theorem \ref{thm2.2} are optimal.

\begin{thm}\label{thm2.3}
  Suppose \eqref{2.4}--\eqref{2.6} and \eqref{2.13} hold,
\[  
T>0,\quad 0<N_1 <M_1,\quad and \quad 0<N_2 <M_2,
\]
where $M_1$ and $M_2$ are defined in \eqref{2.17} and \eqref{2.18}.
Then there exist solutions
\[
u\in Y^{p}_{\alpha}\cap C(\mathbb{R}^n \times\mathbb{R}) \quad \text{and}\quad
v\in Y^{q}_{\beta}\cap C(\mathbb{R}^n \times\mathbb{R})
\] of
\eqref{2.2}, \eqref{2.3} such that for $0<t<T$ we have
	$$u(0,t)\geq
        N^{\frac{\lambda}{1-\lambda\sigma}}_{2}t^{\gamma_2 /\sigma}
\quad\text{ and }\quad v(0,t)\geq N^{\frac{\sigma}{1-\lambda\sigma}}_{1}t^{\gamma_1 /\lambda}$$
	where $\gamma_1$ and $\gamma_2$ are defined in \eqref{2.16}.
\end{thm}

Although the estimates \eqref{2.14} and \eqref{2.15} are optimal there
still remains the question as to whether there is a \emph{single}
solution pair $u,v$ which has the same size as these estimates as
$t\to\infty$.  By the following theorem there is such a solution pair.

\begin{thm}\label{thm2.4}
  Suppose \eqref{2.4}--\eqref{2.6} and \eqref{2.13} hold.  Then there
  exist $N>0$ and solutions 
\[u\in Y^{p}_{\alpha} \quad\text{and}\quad
  v\in Y^{q}_{\beta}\] 
of \eqref{2.2}, \eqref{2.3} such that for
  $|x|^2 <t$ we have
  \[
u(x,t)\geq Nt^{\gamma_2 /\sigma}\quad\text{ and }\quad v(x,t)\geq
  Nt^{\gamma_1 /\lambda}
\]
  where $\gamma_1$ and $\gamma_2$ are defined in \eqref{2.16}.
\end{thm}

According to the following theorem, if $(\lambda,\sigma)\in C\cup D$
then there exist bounds as $t\to0^+$ for solutions of
\eqref{2.1}--\eqref{2.3} in neither the pointwise (i.e. $L^\infty$)
sense nor in the $L^r$ sense for certain values of $r$.  Moreover by
Theorem \ref{thm2.6} the same is true as $t\to\infty$.

\begin{thm}\label{thm2.5}
	Suppose \eqref{2.4}--\eqref{2.7} and \eqref{2.8} hold,
	\begin{equation}\label{2.18.1}
	 \sigma>\mu(\lambda),
	\end{equation}
	$r\in(p,\infty]$, and $s\in(s_0 ,\infty]$ where
	\[s_0 =\max\{q,q\sigma_0/\sigma\}=
	  \begin{cases}
	   q\sigma_0/\sigma & \text{if }\sigma\leq\sigma_0\\
	   q & \text{if }\sigma>\sigma_0 .
	  \end{cases}\]
	Then there exist solutions
	\begin{equation}\label{2.19}
	 u\in Y^{p}_{\alpha}\quad\text{ and }\quad v\in Y^{q}_{\beta}
	\end{equation}
	of the initial value problem \eqref{2.2}, \eqref{2.3} and a
        sequence $\{t_j \}\subset(0,1)$ such that
        $\lim_{j\to\infty}t_j =0$ and
	\begin{equation}\label{2.20}
	 \| v^\lambda \|_{L^r (R_j )}=\|(\partial_t -\Delta)^\alpha u\|_{L^r (R_j )}=\infty
	\end{equation}
	\begin{equation}\label{2.21}
	 \| u^\sigma \|_{L^s (R_j )}=\|(\partial_t -\Delta)^\beta v\|_{L^s (R_j )}=\infty
	\end{equation}
	for $j=1,2,\dots$, where
	\begin{equation}\label{2.22}
	 R_j =\{(x,t)\in\mathbb{R}^n \times\mathbb{R}:|x|<\sqrt{t_j}\text{ and }t_j <t<2t_j \}.
	\end{equation}
\end{thm}

Since \eqref{2.19} implies \eqref{2.20} and \eqref{2.21} are not true
for $r=p$ and $s=q$ respectively, we see that \eqref{2.20} is optimal
when $(\lambda,\sigma)\in C\cup D$ and \eqref{2.21} is optimal when
$s_0 =q$ (i.e. when $(\lambda,\sigma)\in D$).

\begin{thm}\label{thm2.6}
  Suppose \eqref{2.4}--\eqref{2.7}, \eqref{2.8}, and \eqref{2.18.1}
  hold.  Let
	\[r_0
        =\frac{(n+2)(\lambda\sigma-1)}{2(\beta+\alpha\sigma)\lambda}\quad\text{
          and }\quad s_0
        =\frac{(n+2)(\lambda\sigma-1)}{2(\alpha+\beta\lambda)\sigma}.\]
        Then $r_0 >p$, $s_0 >q$, and for each 
\[r\in[r_0 ,\infty] \quad\text{and}\quad
        s\in[s_0 ,\infty]\] 
there exist solutions
        $u\in Y^{p}_{\alpha}$ and $v\in Y^{q}_{\beta}$ of the initial
        value problem \eqref{2.2}, \eqref{2.3} and a sequence
        $\{t_j \}\subset(1,\infty)$ such that
        $\lim_{j\to\infty}t_j =\infty$ and $u$ and $v$ satisfy
        \eqref{2.20} and \eqref{2.21} for $j=1,2,\dots$, where
        $R_j$ is given by \eqref{2.22}.
\end{thm}

The following theorem deals with solutions of \eqref{2.1}--\eqref{2.3}
when \eqref{2.4}--\eqref{2.7} and \eqref{2.9} hold.

\begin{thm}\label{thm2.7}
  Suppose \eqref{2.1}--\eqref{2.7} and \eqref{2.9} hold.  Then the
  following statements are true.
	\begin{enumerate}
		\item[(i)] If $\lambda\sigma\geq1$ then
		$u=v=0 \text{ a.e. in } \mathbb{R}^n \times\mathbb{R}$.
		\item[(ii)] If $\lambda\sigma<1$ then 
		$u$ and $v$ satisfy \eqref{2.14} and \eqref{2.15} for all $T>0$.
	\end{enumerate}
\end{thm}

Clearly Theorem \ref{thm2.7}(i) is optimal and the optimality of
Theorem \ref{thm2.7}(ii) follows from Theorems \ref{thm2.3} and
\ref{thm2.4} because neither \eqref{2.8} nor \eqref{2.9} is assumed in
those theorems.

\section{$J_\alpha$ version of results}\label{sec3}
In order to prove our results in Section \ref{sec2}, we will first reformulate
them in terms of the inverse $J_\alpha$ of the fractional heat
operator \eqref{I.3} as follows:

If \eqref{2.4}--\eqref{2.6} hold then by properties (P1)--(P3) in
Section \ref{sec1} of $J_\alpha$ and the definition of the fractional
heat operator \eqref{I.3}, $u$ and $v$ satisfy
\eqref{2.1}--\eqref{2.3} if and only if
\[f:=(\partial_t -\Delta)^\alpha u \quad\text{and}\quad
g:=(\partial_t -\Delta)^\beta v\] 
satisfy
\begin{equation}\label{3.1}
 f\in X^p \quad\text{ and }\quad g\in X^q
\end{equation}
\begin{equation}\label{3.2}
 \begin{rcases}
  0\leq f\leq(J_\beta g)^\lambda \\
  0\leq g\leq(J_\alpha f)^\sigma
 \end{rcases}
 \quad\text{ in }\mathbb{R}^n \times\mathbb{R}
\end{equation}
\begin{equation}\label{3.3}
 f=g=0\quad\text{ in }\mathbb{R}^n \times(-\infty,0).
\end{equation}
In problem \eqref{3.1}--\eqref{3.3}, $f$ and $g$ are nonnegative
functions in $\mathbb{R}^n \times\mathbb{R}$ and thus $J_\alpha f$ and
$J_\beta g$ are well-defined nonnegative extended real valued
functions in $\mathbb{R}^n \times\mathbb{R}$ without assuming
\eqref{2.5} and \eqref{2.6}.  Hence in this section we study the
problem \eqref{3.1}--\eqref{3.3} without assumptions \eqref{2.5} and
\eqref{2.6}.  However our results in this section for the problem
\eqref{3.1}--\eqref{3.3} will only yield corresponding results for the
problem \eqref{2.1}--\eqref{2.3} when \eqref{2.5} and \eqref{2.6} hold,
for otherwise the fractional heat operators in \eqref{2.2} are not
defined.  (For a more detailed discussion of the properties of
$J_\alpha$ when \eqref{2.5} and \eqref{2.6} do not hold see
\cite[Section 4]{T}.)

Actually in this section we will consider solutions
\begin{equation}\label{1.1}
 f\in X^p \quad\text{ and }\quad g\in X^q
\end{equation}
of the following slightly more general version of \eqref{3.2}--\eqref{3.3}:
\begin{equation}\label{1.2}
 \begin{rcases}
  0\leq f\leq K_1 (J_\beta g)^\lambda \\
  0\leq g\leq K_2 (J_\alpha f)^\sigma
 \end{rcases}
 \quad\text{ in }\mathbb{R}^n \times\mathbb{R},\,n\geq1,
\end{equation}
\begin{equation}\label{1.3}
 f=g=0\quad\text{ in }\mathbb{R}^n \times(-\infty,0),
\end{equation}
where
\begin{equation}\label{1.4}
 p,q\in[1,\infty)\quad\text{ and }\quad\lambda,\sigma,\alpha,\beta,K_1 ,K_2 \in(0,\infty)
\end{equation}
are constants.

As in Section \ref{sec2}, we can assume without loss of generality that
\begin{equation}\label{1.5}
 (n+2-2p\alpha)q^2 \sigma\leq(n+2-2q\beta)p^2 \lambda.
\end{equation}

Under the equivalence of problems \eqref{2.1}--\eqref{2.3} and
\eqref{3.1}--\eqref{3.3} discussed above, the following Theorems
\ref{thm3.1}--\ref{thm3.7} when restricted to the case that
\eqref{2.5} and \eqref{2.6} hold and $K_1 =K_2 =1$, clearly
imply Theorems \ref{thm2.1}--\ref{thm2.7} in Section \ref{sec2}.  We
will prove Theorems \ref{thm3.1}--\ref{thm3.7} in Section \ref{sec5}.

If \eqref{1.4} holds then either
\begin{equation}\label{1.6.1}
 2p\alpha<n+2
\end{equation}
or
\begin{equation}\label{1.6.2}
 2p\alpha\geq n+2.
\end{equation}

\begin{thm}\label{thm3.1}
 Suppose \eqref{1.1}--\eqref{1.5} and \eqref{1.6.1} hold and
 $$\frac{1}{\lambda}\leq\sigma<\mu(\lambda).$$
 Then
 $$f=g=0\quad\text{ a.e. in }\mathbb{R}^n \times\mathbb{R}.$$
\end{thm}

\begin{thm}\label{thm3.2}
 Suppose \eqref{1.1}--\eqref{1.5} and \eqref{1.6.1} hold and
 \begin{equation}\label{1.6.3}
  \sigma< 1/\lambda.
 \end{equation}
 Then for all $T>0$ we have
 \begin{equation}\label{1.7}
  \| f\|_{L^\infty (\mathbb{R}^n \times(0,T))}\leq(K_1 K^{\lambda}_{2})^{\frac{1}{1-\lambda\sigma}}M^{\frac{\lambda\sigma}{1-\lambda\sigma}}_{1}T^{\gamma_1},
 \end{equation}
 \begin{equation}\label{1.8}
  \| g\|_{L^\infty (\mathbb{R}^n \times(0,T))}\leq(K_2 K^{\sigma}_{1})^{\frac{1}{1-\lambda\sigma}}M^{\frac{\lambda\sigma}{1-\lambda\sigma}}_{2}T^{\gamma_2},
 \end{equation}
 \begin{equation}\label{1.9}
  \| J_\alpha f\|_{L^\infty (\mathbb{R}^n \times(0,T))}\leq(K_1
  K^{\lambda}_2
)^{\frac{1}{1-\lambda\sigma}}M^{\frac{\lambda}{1-\lambda\sigma}}_{2}T^{\gamma_2 /\sigma},
 \end{equation}
 and
 \begin{equation}\label{1.10}
  \| J_\beta g\|_{L^\infty (\mathbb{R}^n \times(0,T))}\leq(K_2 K^{\sigma}_1)^{\frac{1}{1-\lambda\sigma}}M^{\frac{\sigma}{1-\lambda\sigma}}_{1}T^{\gamma_1/\lambda},
 \end{equation}
 where $\gamma_1 ,\gamma_2 ,M_1$, and $M_2$ defined in \eqref{2.16}--\eqref{2.18}.
\end{thm}
 
\begin{thm}\label{thm3.3}
 Suppose \eqref{1.4} and \eqref{1.6.3} hold,
 \begin{equation}\label{1.15}
  T>0, \quad 0<N_1 <M_1, \quad\text{ and }\quad 0<N_2 <M_2,
 \end{equation}
 where $M_1$ and $M_2$ are defined in \eqref{2.17} and \eqref{2.18}.  Then there exist solutions
 \begin{equation}\label{1.16}
 f\in L^p (\mathbb{R}^n \times\mathbb{R})\cap C(\mathbb{R}^n
 \times\mathbb{R})\quad\text{ and }\quad g\in L^q (\mathbb{R}^n \times\mathbb{R})\cap C(\mathbb{R}^n \times\mathbb{R})
 \end{equation}
 of \eqref{1.2}, \eqref{1.3} such that
 \begin{equation}\label{1.17}
 J_\alpha f,\,J_\beta g\in C(\mathbb{R}^n \times\mathbb{R})
 \end{equation}
 and, for $0<t<T$,
 \begin{equation}\label{1.18}
  f(0,t)=(K_1 K^{\lambda}_{2})^{\frac{1}{1-\lambda\sigma}}N^{\frac{\lambda\sigma}{1-\lambda\sigma}}_{1}t^{\gamma_1},
 \end{equation}
  \begin{equation}\label{1.19}
  g(0,t)=(K_2 K^{\sigma}_{1})^{\frac{1}{1-\lambda\sigma}}N^{\frac{\lambda\sigma}{1-\lambda\sigma}}_{2}t^{\gamma_2},
 \end{equation}
  \begin{equation}\label{1.20}
  J_\alpha f(0,t)\geq(K_1 K^{\lambda}_{2})^{\frac{1}{1-\lambda\sigma}}N^{\frac{\lambda}{1-\lambda\sigma}}_{2}t^{\gamma_2 /\sigma},
 \end{equation}
 \begin{equation}\label{1.21}
  J_\beta g(0,t)\geq(K_2 K^{\sigma}_{1})^{\frac{1}{1-\lambda\sigma}}N^{\frac{\sigma}{1-\lambda\sigma}}_{1}t^{\gamma_1 /\lambda},
 \end{equation}
 where $\gamma_1$ and $\gamma_2$ are defined in \eqref{2.16}.
\end{thm}

\begin{thm}\label{thm3.4}
  Suppose \eqref{1.4} and \eqref{1.6.3} hold.  Then there exist $N>0$
  and solutions $f\in X^p$ and $g\in X^q$ of \eqref{1.2}, \eqref{1.3}
  such that for $|x|^2 <t$ we have
 \begin{equation}\label{1.22}
  f(x,t)\geq Nt^{\gamma_1},\quad g(x,t)\geq Nt^{\gamma_2},
 \end{equation}
 \begin{equation}\label{1.23}
  J_\alpha f(x,t)\geq Nt^{\gamma_2 /\sigma},\quad\text{ and }\quad J_\beta g(x,t)\geq Nt^{\gamma_1 /\lambda},
 \end{equation}
 where $\gamma_1$ and $\gamma_2$ are defined in \eqref{2.16}.
\end{thm}

\begin{thm}\label{thm3.5}
 Suppose \eqref{1.4}, \eqref{1.5}, and \eqref{1.6.1} hold,
 \begin{equation}\label{T6.1}
  \sigma>\mu(\lambda),
 \end{equation}
 $r\in(p,\infty]$, and $s\in(s_0 ,\infty]$, where
 $$s_0 =\max\{q,q\sigma_0/\sigma\}=
 \begin{cases}
  q\sigma_0 /\sigma & \text{if }\sigma<\sigma_0 \\
  q & \text{if }\sigma\geq\sigma_0 .
 \end{cases}
 $$  
 Then there exist solutions
 \begin{equation}\label{T6.3}
  f\in L^p (\mathbb{R}^n \times\mathbb{R})\quad\text{ and }\quad g\in L^q (\mathbb{R}^n \times\mathbb{R})
 \end{equation}
 of the initial value problem \eqref{1.2}, \eqref{1.3} and a sequence
 $\{t_j \}\subset(0,1)$ such that $\lim_{j\to\infty}t_j =0$ and
 \begin{equation}\label{T6.4}
  \| f\|_{L^r (R_j )}=\| g\|_{L^s (R_j )}=\infty\quad\text{ for }j=1,2,\dots
 \end{equation}
 where
 \begin{equation}\label{T6.5}
  R_j =\{(x,t)\in\mathbb{R}^n \times\mathbb{R}:|x|<\sqrt{t_j}\text{ and }t_j <t<2t_j \}.
 \end{equation}
\end{thm}

\begin{thm}\label{thm3.6}
 Suppose \eqref{1.4}, \eqref{1.5}, \eqref{1.6.1} and \eqref{T6.1} hold.  Let
 $$r_0
 =\frac{(n+2)(\lambda\sigma-1)}{2(\beta+\alpha\sigma)\lambda}\quad\text{
   and }\quad s_0 =\frac{(n+2)(\lambda\sigma-1)}{2(\alpha+\beta\lambda)\sigma}.$$
 Then $r_0 >p$, $s_0 >q$, and for each
 \begin{equation}\label{T7.1}
  r\in[r_0 ,\infty]\quad\text{ and }\quad s\in[s_0 ,\infty]
 \end{equation}
 there exist solutions
 \begin{equation}\label{T7.2}
  f\in X^p \quad\text{ and }\quad g\in X^q
 \end{equation}
 of the initial value problem \eqref{1.2}, \eqref{1.3} and a sequence
 $\{t_j \}\subset(1,\infty)$ such that $\lim_{j\to\infty}t_j =\infty$
 and $f$ and $g$ satisfy \eqref{T6.4} where $R_j$ is given by
 \eqref{T6.5}.
\end{thm}

\begin{thm}\label{thm3.7}
  Suppose \eqref{1.1}--\eqref{1.5} and \eqref{1.6.2} hold.  Then the
  following statements are true.
 \begin{enumerate}
  \item[(i)] If $\lambda\sigma\geq1$ then
  $f=g=0\text{ a.e. in }\mathbb{R}^n \times\mathbb{R}$.
  \item[(ii)] If $\lambda\sigma<1$ then
  $f$ and $g$ satisfy \eqref{1.7}--\eqref{1.10} for all $T>0$.
 \end{enumerate}
\end{thm}

\section{Preliminarys}\label{sec4}
In the section we provide some remarks and lemmas needed for the
proofs of our results in Section \ref{sec3} dealing with solutions of the
$J_\alpha$ problem \eqref{1.1}--\eqref{1.4}.

\begin{rem}\label{rem4.1}
  If \eqref{1.4}$_2$ holds and $\lambda\sigma<1$ then the functions
  $F,G:\mathbb{R}^n \times\mathbb{R}\to[0,\infty)$ defined in
  $\mathbb{R}^n \times(-\infty,0]$ by $F=G=0$ and defined for
  $(x,t)\in\mathbb{R}^n \times(0,\infty)$ by
 $$F(x,t)=F(t)=M^{\frac{\lambda\sigma}{1-\lambda\sigma}}_{1}t^{\gamma_1}\quad\text{
   and }\quad G(x,t)=G(t)=M^{\frac{\lambda\sigma}{1-\lambda\sigma}}_{2}
 t^{\gamma_2},$$
 where $\gamma_1 ,\gamma_2 ,M_1$, and $M_2$ are defined in
 \eqref{2.16}--\eqref{2.18}, satisfy
 \begin{equation}\label{8.1}
  F=(J_\beta G)^\lambda \quad\text{ and }\quad G=(J_\alpha F)^\sigma \quad\text{ in }\mathbb{R}^n \times\mathbb{R},
 \end{equation}
 which can be verified using the formula
\begin{equation}\label{L3.5}
\int^{t}_{0}\frac{(t-\tau)^{\alpha-1}\tau^{\beta-1}}{\Gamma(\alpha)\Gamma(\beta)}\,d\tau=\frac{t^{\alpha+\beta-1}}{\Gamma(\alpha+\beta)}\quad\text{ for }t, 
 \alpha,\beta>0.
\end{equation}
Even though $F,G\notin X^p$ for all $p\geq1$, these functions will be
useful in our analysis of solutions of \eqref{1.2}, \eqref{1.3} which
are in $X^p$ for some $p\geq1$.
 \end{rem}

\begin{rem}\label{rem4.2}
  It will be convenient to scale \eqref{1.2} as follows.  Suppose
  \eqref{1.4}$_2$ holds, $\lambda\sigma\neq 1$, $T>0$, and
  $f,g,\bar{f},\bar{g}:\mathbb{R}^n \times\mathbb{R}\to\mathbb{R}$ are nonnegative
  measurable functions such that $f=g=\bar{f}=\bar{g}=0$ in
  $\mathbb{R}^n \times(-\infty,0)$ and
 $$f(x,t)=(K_1K^{\lambda}_{2})^{\frac{1}{1-\lambda\sigma}}T^{\gamma_1}\bar{f}(\bar{x},\bar{t})\quad\text{and }\quad g(x,t)=(K_2 K^{\sigma}_{1})^{\frac{1}{1-\lambda\sigma}}T^{\gamma_2}\bar{g}(\bar{x},\bar{t})$$
 where $\gamma_1$ and $\gamma_2$ are defined in \eqref{2.16} and
 $x=T^{1/2}\bar{x}$ and $t=T\bar{t}$.  Then $f$ and $g$ satisfy
 \eqref{1.2}
 if and only if $\bar{f}$ and $\bar{g}$ satisfy
  \[\begin{rcases}
     0\leq\bar{f}\leq(J_\beta\bar{g})^\lambda \\
     0\leq\bar{g}\leq(J_\alpha\bar{f})^\sigma
    \end{rcases}
    \quad\text{ in }\mathbb{R}^n \times\mathbb{R}.\]
    Moreover, for $(x,t)\in\mathbb{R}^n \times(0,\infty)$ we have
\begin{align*}
 &\frac{f(x,t)}{t^{\gamma_1}}=(K_1K^{\lambda}_{2})^{\frac{1}{1-\lambda\sigma}}\frac{\bar{f}(\bar{x},\bar{t})}{\bar{t}^{\gamma_1}},\qquad\frac{g(x,t)}{t^{\gamma_2}}=(K_2 K^{\sigma}_{1})^{\frac{1}{1-\lambda\sigma}}\frac{\bar{g}(\bar{x},\bar{t})}{\bar{t}^{\gamma_2}},\\
 &\frac{J_\alpha f(x,t)}{t^{\gamma_2 /\sigma}}=(K_1 K^{\lambda}_{2})^{\frac{1}{1-\lambda\sigma}}\frac{J_\alpha \bar{f}(\bar{x},\bar{t})}{\bar{t}^{\gamma_2 /\sigma}}\quad\text{ and }\quad\frac{J_\beta g(x,t)}{t^{\gamma_1 /\lambda}}=(K_2 K^{\sigma}_{1})^{\frac{1}{1-\lambda\sigma}}\frac{J_\beta \bar{g}(\bar{x},\bar{t})}{\bar{t}^{\gamma_1 /\lambda}}.
\end{align*}
\end{rem}

\begin{lem}\label{lem4.1}
 Suppose \eqref{1.1}--\eqref{1.4} and \eqref{1.6.2} hold.  Then
 \begin{equation}\label{L1.2}
  f,g\in X^\infty .
 \end{equation}
\end{lem}

\begin{proof}
 Let $T>0$ be fixed.  To prove \eqref{L1.2} it suffices by \eqref{1.3} to prove
 $$f,g\in L^\infty (\mathbb{R}^n \times(0,T)).$$
 Choose
 \begin{equation}\label{L1.3}
  p_1 >\max\{p,\sigma,\frac{(n+2)\sigma}{2\beta}\}.
 \end{equation}
 Define $\widehat\alpha\in\mathbb{R}$ by
 $$\frac{2\widehat{\alpha}}{n+2}+\frac{1}{2p_1}=\frac{1}{p}\leq\frac{2\alpha}{n+2}$$
 by \eqref{1.6.2}.  Then $0<\widehat{\alpha}<\alpha$ and
 $$0<\frac{1}{p}-\frac{1}{p_1}=\frac{2\widehat{\alpha}}{n+2}-\frac{1}{2p_1}<\frac{2\widehat{\alpha}}{n+2}<\frac{1}{p}\leq1$$
 by \eqref{1.4}$_1$.  Hence by \eqref{1.3}, \eqref{1.1}$_1$, and Lemma
 \ref{lem7.2} we have
 $$(J_\alpha f)|_{\mathbb{R}^n \times(0,T)}\leq C(J_{\widehat{\alpha}}f)|_{\mathbb{R}^n \times(0,T)}\in L^{p_1}(\mathbb{R}^n \times(0,T)).$$
 Consequently by \eqref{1.2}$_2$ we have
 \begin{equation}\label{L1.4}
  g\in L^{p_1 /\sigma}(\mathbb{R}^n \times(0,T)).
 \end{equation}
 By \eqref{L1.3} there exists $\widehat{\beta}\in(0,\beta]$ such that
 $$\frac{\sigma}{p_1}<\frac{2\widehat{\beta}}{n+2}<1.$$
 Then by \eqref{1.2}$_1$, \eqref{1.3}, \eqref{L1.4}, Lemma \ref{lem7.2} we have 
 $$f^{1/\lambda}|_{\mathbb{R}^n \times(0,T)}\leq C(J_\beta g)|_{\mathbb{R}^n \times(0,T)}\leq C(J_{\widehat{\beta}}g)|_{\mathbb{R}^n \times(0,T)}\in L^\infty (\mathbb{R}^n \times(0,T)).$$
 Thus by \eqref{1.2}$_2$, \eqref{1.3}, and Lemma \ref{lem7.1} we have
 $$g\in L^\infty (\mathbb{R}^n \times(0,T)).$$
\end{proof}

\begin{lem}\label{lem4.2}
 Suppose \eqref{1.1}--\eqref{1.4} and \eqref{1.6.1} hold,
 \begin{equation}\label{L2.1}
  \sigma<\mu(\lambda),
 \end{equation}
 \begin{equation}\label{L2.2}
  p_1 \in[p,\infty),\text{ and }f\in X^{p_1}.
 \end{equation}
 Then either
 \begin{equation}\label{L2.3}
  f\in X^\infty
 \end{equation}
 or there exists a constant
 \begin{equation}\label{L2.4}
  C_0 =C_0 (n,\lambda,\sigma,\alpha,\beta,p)>0
 \end{equation}
 such that $f\in X^{p_2}$ for some $p_2 \in(p_1 ,\infty)$ satisfying
 \begin{equation}\label{L2.5}
  \frac{1}{p_1}-\frac{1}{p_2}>C_0 .
 \end{equation}
\end{lem}

\begin{proof}
  If $2p_1 \alpha\geq n+2$ then \eqref{L2.3} follows from Lemma
  \ref{lem4.1} with $p=p_1$.  Hence we can assume
 \begin{equation}\label{L2.8}
  2p_1 \alpha<n+2.
 \end{equation}
 By \eqref{1.4}, \eqref{1.6.1}, and \eqref{L2.1} there exists
 $$\varepsilon=\varepsilon(n,\lambda,\sigma,\alpha,\beta, p)>0$$
 such that
 \begin{equation}\label{L2.9}
  \alpha_\varepsilon :=\alpha-\varepsilon>0,\quad \beta_\varepsilon :=\beta-\varepsilon>0
 \end{equation}
 and
 \begin{equation}\label{L2.10}
  \sigma<\frac{2p\beta_\varepsilon}{n+2-2p\alpha_\varepsilon}+\frac{n+2}{(n+2-2p\alpha_\varepsilon )\lambda}.
 \end{equation}
 By \eqref{L2.8} and \eqref{L2.9} we have
 \begin{equation}\label{L2.11}
  n+2-2p_1 \alpha_\varepsilon >2p_1 \varepsilon .
 \end{equation}
 By \eqref{L2.8}, \eqref{L2.9}, \eqref{1.4}, and \eqref{L2.2} there
 exists $p_3 \in(p_1 ,\infty)$ such that
 \begin{equation}\label{L2.12}
  \frac{1}{p_1}-\frac{1}{p_3}=\frac{2\alpha_\varepsilon}{n+2}<\frac{2\alpha}{n+2}<\frac{1}{p_1}\leq1.
 \end{equation}
 Hence by \eqref{1.3}, \eqref{L2.2}, and Lemma \ref{lem7.2} we have
 $J_\alpha f\in X^{p_3}$ and thus from \eqref{1.2} and \eqref{1.1} we find that
 \begin{equation}\label{L2.13}
  g\in X^{p_4}\text{ where }p_4 =\max\{q,p_3 /\sigma\}\geq1
 \end{equation}
 by \eqref{1.4}.  We can assume
 \begin{equation}\label{L2.14}
  2p_4 \beta<n+2
 \end{equation}
 for otherwise by Lemma \ref{lem4.1} with $q=p_4$ and the roles of
 $(f,p,\lambda,\alpha)$ and $(g,q,\sigma,\beta)$ interchanged, we have
 \eqref{L2.3} holds.
 
 It follows from 
\eqref{L2.14}, \eqref{L2.9}, and \eqref{L2.13}
that there exists
 $p_5 \in(p_4 ,\infty)$ such that
 \begin{equation}\label{L2.15}
  \frac{1}{p_4}-\frac{1}{p_5}=\frac{2\beta_\varepsilon}{n+2}<\frac{2\beta}{n+2}<\frac{1}{p_4}\leq1.
 \end{equation}
 Thus by \eqref{1.3}, \eqref{L2.13} and Lemma \ref{lem7.2} we have
 $J_\beta g\in X^{p_5}$ and consequently by \eqref{1.2}
 \begin{equation}\label{L2.16}
  f\in X^{p_2}\text{ where }p_2 =p_5 /\lambda.
 \end{equation}
 Moreover, it follows from \eqref{L2.16}, \eqref{L2.15},
 \eqref{L2.13}, \eqref{L2.12}, \eqref{L2.11}, and \eqref{L2.2}$_1$
 that
 \begin{align*}
  \frac{1}{p_1}-\frac{1}{p_2}&=\frac{1}{p_1}-\frac{\lambda}{p_5}=\frac{1}{p_1}-\lambda\left(\frac{1}{p_4}-\frac{2\beta_\varepsilon}{n+2}\right)\\
  &\ge\frac{1}{p_1}-\lambda\left(\frac{\sigma}{p_3}-\frac{2\beta_\varepsilon}{n+2}\right)\\
  &=\frac{1}{p_1}-\lambda\left(\sigma\left(\frac{1}{p_1}-\frac{2\alpha_\varepsilon}{n+2}\right)-\frac{2\beta_\varepsilon}{n+2}\right)\\
  &=\frac{\lambda}{(n+2)p_1}\left[\frac{n+2}{\lambda}-(\sigma(n+2-2\alpha_\varepsilon p_1 )-2\beta_\varepsilon p_1)\right]\\
  &=\frac{\lambda(n+2-2\alpha_\varepsilon p_1)}{(n+2)p_1}\left[\frac{2\beta_\varepsilon p_1}{n+2 -2\alpha_\varepsilon p_1}+\frac{n+2}{(n+2-2\alpha_\varepsilon p_1 )\lambda}-\sigma\right]\\
  &\geq\frac{2\lambda\varepsilon}{n+2}\left[\frac{2p\beta_\varepsilon}{n+2-2p\alpha_\varepsilon}+\frac{n+2}{(n+2 -2p\alpha_\varepsilon )\lambda}-\sigma\right]\\
  &=C_0 (n,\lambda,\sigma,\alpha,\beta,p)>0
 \end{align*}
 by \eqref{L2.10}. 
\end{proof} 

\begin{lem}\label{lem4.3}
 Suppose \eqref{1.1}--\eqref{1.4} and \eqref{1.6.1} hold and $\sigma<\mu(\lambda)$.  Then
 \begin{equation}\label{L2.17}
  f,g\in X^\infty .
 \end{equation}
\end{lem}

\begin{proof}
  Starting with the assumption that $f$ satisfies \eqref{1.1}$_1$ and
  iterating Lemma \ref{lem4.2} a finite number of times ($m$ times is enough if
  $m>1/(pC_0 )$) we find that $f\in X^\infty$ and hence \eqref{L2.17}
  follows from \eqref{1.2}, \eqref{1.3}, and Lemma \ref{lem7.1}.
\end{proof}

\begin{lem}\label{lem4.4}
  Suppose $f,g:\mathbb{R}^n \times\mathbb{R}\to\mathbb{R}$ are
  nonnegative measurable functions satisfying \eqref{1.2} and
  \eqref{1.3} where
 $$\lambda,\sigma,\alpha,\beta,K_1 ,K_2 \in(0,\infty)$$
 and for some $h\in\{f,g\}$ we have
 \begin{equation}\label{L3.0}
  h\in X^\infty \quad\text{ and }\quad\| h\|_{L^\infty (\mathbb{R}^n \times\mathbb{R})}\neq0.
 \end{equation}
 Then $\lambda\sigma<1$ and $f$ and $g$ satisfy
 \eqref{1.7}--\eqref{1.10} for all $T>0$.
\end{lem}

\begin{proof}
  
  Suppose $h=f$.  The proof when $h=g$ is similar and will be omitted.
  By \eqref{1.2} we have
 \begin{equation}\label{L3.1}
  0\leq f\leq
K_1K_2^\lambda
(J_\beta ((J_\alpha f)^\sigma ))^\lambda \quad\text{ in }\mathbb{R}^n \times\mathbb{R}
 \end{equation}
 and by \eqref{1.3} and \eqref{L3.0} there exists $a\geq0$ such that
 \begin{equation}\label{L3.2}
  \| f\|_{L^\infty (\mathbb{R}^n \times(-\infty,a])}=0
 \end{equation}
 and
 \begin{equation}\label{L3.3}
  0<\| f\|_{L^\infty (\mathbb{R}^n \times(-\infty,t))}<\infty \quad\text{ for }t>a.
 \end{equation}
 Thus
 \begin{equation}\label{L3.4}
  J_\alpha f=0\quad\text{ in }\mathbb{R}^n \times(-\infty,a].
 \end{equation}
 
 Let $T>a$ be momentarily fixed.  Then for
 $(x,t)\in\mathbb{R}^n \times(a,T]$ we find from \eqref{L3.2} and
 \eqref{L3.3} that
 $$\frac{(J_\alpha f)(x,t)}{\| f\|_{L^\infty (\mathbb{R}^n \times(a,T))}}\leq\int^{t}_{a}\frac{(t-\tau)^{\alpha-1}}{\Gamma(\alpha)}\,d\tau=\frac{(t-a)^\alpha}{\Gamma(\alpha+1)}.$$
 Hence for $(x,t)\in\mathbb{R}^n \times(a,T]$ we obtain from
 \eqref{L3.4} and \eqref{L3.5}that
 \[\frac{J_\beta ((J_\alpha f)^\sigma )(x,t)}{\| f\|^{\sigma}_{L^\infty (\mathbb{R}^n \times(a,T))}}\leq\frac{1}{\Gamma(\alpha+1)^\sigma}\int^{t}_{a}\frac{(t-\tau)^{\beta-1}}{\Gamma(\beta)}(\tau-a)^{\alpha\sigma}\,d\tau=B(t-a)^{\beta+\alpha\sigma}\]
 where
 $$B=\frac{\Gamma(\alpha\sigma+1)}{\Gamma(\alpha+1)^\sigma \Gamma(\beta+\alpha\sigma+1)}.$$
Thus by \eqref{L3.1} we see that
 \begin{align*}
  \| f\|_{L^\infty (\mathbb{R}^n \times(a,T))}&\leq
K_1K_2^\lambda
\|(J_\beta ((J_\alpha f)^\sigma ))^\lambda \|_{L^\infty (\mathbb{R}^n \times(a,T))}\\
  &\leq K_1K_2^\lambda B^\lambda \| f\|_{L^\infty (\mathbb{R}^n \times(a,T))}^{\lambda\sigma}(T-a)^{\lambda(\beta+\alpha\sigma)}\quad\text{for }T>a,
 \end{align*}
 which by \eqref{L3.3} implies
 \begin{equation}\label{L3.6}
  1\leq K_1K_2^\lambda B^\lambda \| f\|^{\lambda\sigma-1}_{L^\infty (\mathbb{R}^n \times(a,T))}(T-a)^{\lambda(\beta+\alpha\sigma)}\quad\text{ for }T>a.
 \end{equation}
 Thus $\lambda\sigma<1$ for otherwise sending $T$ to $a$ in
 \eqref{L3.6} gives a contradiction. Hence from \eqref{L3.6}
 and \eqref{2.16}$_1$ we get
 \[\| f\|_{L^\infty (\mathbb{R}^n \times(a,T))}\leq 
(K_1K_2^\lambda)^{\frac{1}{1-\lambda\sigma}}
B^{\frac{\lambda}{1-\lambda\sigma}}(T-a)^{\gamma_1}\quad\text{ for }T>a\]
 which together with \eqref{L3.2} and the nonnegativity of $a$ implies
 \begin{equation}\label{L3.7}
  \| f\|_{L^\infty (\mathbb{R}^n \times(0,t))}\leq (K_1K_2^\lambda)^{\frac{1}{1-\lambda\sigma}}B^{\frac{\lambda}{1-\lambda\sigma}}t^{\gamma_1}\quad\text{ for }t>0.
 \end{equation}

 Suppose for some $\delta>0$ we have
 \begin{equation}\label{L3.8}
  \| f\|_{L^\infty (\mathbb{R}^n \times(0,t))}\leq\delta t^{\gamma_1}\quad\text{ for }t>0.
 \end{equation}
 Then by \eqref{L3.2} and \eqref{L3.5} we have for
 $(x,t)\in\mathbb{R}^n \times(0,\infty)$ that
 \begin{equation}\label{L3.9}
  (J_\alpha f)(x,t)\leq\delta\int^{t}_{0}\frac{(t-\tau)^{\alpha-1}}{\Gamma(\alpha)}\tau^{\gamma_1}\,d\tau=\delta\frac{\Gamma(\gamma_1 +1)}{\Gamma(\alpha+\gamma_1 +1)}t^{\alpha+\gamma_1}
 \end{equation}
 and hence by \eqref{L3.4}, \eqref{L3.5}, and \eqref{2.17} we find for
 $(x,t)\in\mathbb{R}^n \times(0,\infty)$ that
 \begin{align*}
  [J_\beta ((J_\alpha f)^\sigma )(x,t)]^\lambda &\leq\left[\delta^\sigma \frac{\Gamma(\gamma_1 +1)^\sigma}{\Gamma(\alpha+\gamma_1 +1)^\sigma}\int^{t}_{0}\frac{(t-\tau)^{\beta-1}}{\Gamma(\beta)}\tau^{(\alpha+\gamma_1 )\sigma}\,d\tau\right]^\lambda \\
  &=\left[\delta^\sigma \frac{\Gamma(\gamma_1 +1)^\sigma}{\Gamma(\alpha+\gamma_1 +1)^\sigma}\,\frac{\Gamma((\alpha+\gamma_1 )\sigma+1)}{\Gamma(\beta+(\alpha+\gamma_1 )\sigma+1)}t^{\beta+(\alpha+\gamma_1 )\sigma}\right]^\lambda \\
  &=(\delta M_1 )^{\lambda\sigma}t^{\gamma_1}
 \end{align*}
 because by \eqref{2.16} we have
 \begin{equation}\label{L3.10}
  (\beta+(\alpha+\gamma_1 )\sigma)\lambda=\gamma_1 \quad\text{ and }\quad(\alpha+\gamma_1 )\sigma=\gamma_2 .
 \end{equation}
 Thus by \eqref{L3.1}
 \begin{equation}\label{L3.11}
  \| f\|_{L^\infty (\mathbb{R}^n \times(0,t))}\leq K_1K_2^\lambda(\delta M_1 )^{\lambda\sigma}t^{\gamma_1}\quad\text{ for }t>0.
 \end{equation}
 Next defining a sequence $\{\delta_j \}\subset(0,\infty)$ by
 \[
\delta_1 =K_1K_2^\lambda B^{\frac{\lambda}{1-\lambda\sigma}}\quad\text{and}\quad
 \delta_{j+1}=K_1K_2^\lambda(\delta_j M_1 )^{\lambda\sigma}
\] 
and using
 $0<\lambda\sigma<1$ we see that
 $\delta_j \to (K_1K_2^\lambda)^{\frac{1}{1-\lambda\sigma}}M^{\frac{\lambda\sigma}{1-\lambda\sigma}}_{1}$ as
 $j\to\infty$.  It therefore follows from \eqref{L3.7}, \eqref{L3.8},
 and \eqref{L3.11} that $f$ satisfies \eqref{1.7} for $T>0$.  Thus
 \eqref{L3.9} holds with
 $\delta=(K_1K_2^\lambda)^{\frac{1}{1-\lambda\sigma}}M^{\frac{\lambda\sigma}{1-\lambda\sigma}}_{1}$ and so from
 \eqref{1.2}$_2$ we find that
 \begin{align*} 
  \| g\|_{L^\infty \times(0,t))}&\leq K_2\left((K_1K_2^\lambda)^{\frac{1}{1-\lambda\sigma}}M^{\frac{\lambda\sigma}{1-\lambda\sigma}}_{1}\frac{\Gamma(\gamma_1 +1)}{\Gamma(\alpha+\gamma_1 +1)}\right)^\sigma t^{\sigma(\alpha+\gamma_1 )}\\
  &=(K_2K_1^\sigma)^{\frac{1}{1-\lambda\sigma}}M^{\frac{\lambda\sigma}{1-\lambda\sigma}}_{2}t^{\gamma_2}\quad\text{ for }t>0
 \end{align*}
 by \eqref{L3.10}, \eqref{2.17}, and \eqref{2.18}.  That is $g$ satisfies \eqref{1.8} for $T>0$.  Finally, \eqref{1.9} and \eqref{1.10} follow from \eqref{1.7}, \eqref{1.8}, and \eqref{L3.5}.
\end{proof}

\section{Proofs of results for $J_\alpha$ problem.}\label{sec5}
In this section we prove our results stated in Section \ref{sec3}
concerning pointwise bounds for nonnegative solutions $f$ and $g$ of
\eqref{1.1}--\eqref{1.4}.  As explained in Section \ref{sec3}, these
results immediately imply Theorems \ref{thm2.1}--\ref{thm2.7} in
Section \ref{sec2}. 

\begin{proof}[Proof of Theorem \ref{thm3.7} (resp. Theorems
  \ref{thm3.1} and \ref{thm3.2})] It follows from Lemma \ref{lem4.1}
  (resp. Lemma \ref{lem4.3}) that $f,g\in X^\infty$.  Hence Theorem
  \ref{thm3.7} (resp. Theorems \ref{thm3.1} and \ref{thm3.2})
  follow(s) from Lemma \ref{lem4.4}.
\end{proof}

\begin{proof}[Proof of Theorem \ref{thm3.3}]
Let
\begin{equation}\label{PT2.1}
 m=\max\left\{\left(\frac{N_2}{M_2}\right)^{\frac{\lambda^2 \sigma}{1-\lambda\sigma}},\left(\frac{N_1}{M_1}\right)^{\frac{\lambda\sigma}{1-\lambda\sigma}}\right\}
\end{equation}
and $a_1 =(m+1)/2$.  By \eqref{1.6.3} and \eqref{1.15} we have $m\in(0,1)$ and thus
\begin{equation}\label{PT2.2}
 0<m<a_1 <1.
\end{equation}
It therefore follows from \eqref{1.6.3} that
$a^{1/\lambda}_{1}<a^{\sigma}_{1}$ and hence there exists $a_2$ such
that
$$0<a^{1/\lambda}_{1}<a_2 <a^{\sigma}_{1}<1$$
which together with \eqref{PT2.1} and \eqref{PT2.2} gives
\begin{equation}\label{PT2.3}
 \frac{a^{\lambda}_{2}}{a_1}>1,\quad\frac{a^{\sigma}_{1}}{a_2}>1,
\end{equation}
\begin{equation}\label{PT2.4}
  a_1
  >\left(\frac{N_1}{M_1}\right)^{\frac{\lambda\sigma}{1-\lambda\sigma}},\quad\text{
    and }\quad a_2 >\left(\frac{N_2}{M_2}\right)^{\frac{\lambda\sigma}{1-\lambda\sigma}}.
\end{equation}

By Remark \ref{rem4.2}, we can assume $K_1 =K_2 =T=1$.  For
$(x,t)\in\mathbb{R}^n \times\mathbb{R}$ and $\delta\in(0,1)$ let
\begin{equation}\label{PT2.5}
 F_\delta (x,t)=F_\delta (t)=\psi_\delta (t)F(t)\quad\text{ and }\quad
 G_\delta (x,t)=G_\delta (t)=\psi_\delta (t)G(t)
\end{equation}
where $F$ and $G$ are as in Remark \ref{rem4.1} and $\psi_\delta \in C^\infty (\mathbb{R}\to[0,1])$ satisfies
$$\psi_\delta (t)=
   \begin{cases}
    1\text{ if }t\leq1\\
    0\text{ if }t\geq1+\delta.
   \end{cases}$$
Then for $1\leq t\leq1+\delta$
\begin{align*}
 J_\alpha F(t)&-J_\alpha F_\delta (t)=\int^{t}_{1}\frac{(t-\tau)^{\alpha-1}}{\Gamma(\alpha)}F(\tau)(1-\psi_\delta (\tau))\,d\tau\\
 &\leq\int^{t}_{1}\frac{(t-\tau)^{\alpha-1}}{\Gamma(\alpha)}F(\tau)\,d\tau\leq F(1+\delta)\int^{t}_{1}\frac{(t-\tau)^{\alpha-1}}{\Gamma(\alpha)}\,d\tau\\
 &=F(1+\delta)\frac{(t-1)^\alpha}{\Gamma(\alpha+1)}\leq F(2)\frac{\delta^\alpha}{\Gamma(\alpha+1)}
\end{align*}
and similarly
$$J_\beta G(t)-J_\beta G_\delta (t)\leq G(2)\frac{\delta^\beta}{\Gamma(\beta+1)}.$$
Thus by \eqref{8.1} we have for $1\leq t\leq1+\delta$ that
\begin{align*}
 \frac{J_\alpha F_\delta (t)}{J_\alpha F(t)}&=\frac{J_\alpha F(t)-(J_\alpha F(t)-J_\alpha F_\delta (t))}{G(t)^{1/\sigma}}\\
 &\geq1-\frac{F(2)\delta^\alpha}{\Gamma(\alpha+1)G(1)^{1/\sigma}}\\
 &=1-C\delta^\alpha 
\end{align*}
and similarly for $1\le t\le 1+\delta$ that
$$\frac{J_\beta G_\delta (t)}{J_\beta G(t)}\geq1-C\delta^\beta$$
where $C=C(\lambda,\sigma,\alpha,\beta)>0$.  Hence choosing
$\delta\in(0,1)$ sufficiently small and using \eqref{8.1} and
\eqref{PT2.3} we find for $1\leq t\leq1+\delta$ that
\begin{equation}\label{PT2.6}
 G_\delta (t)\leq G(t)= (J_\alpha F(t))^\sigma \leq\sqrt{\frac{a^{\sigma}_{1}}{a_2}}(J_\alpha F_\delta (t))^\sigma
\end{equation}
and
\begin{equation}\label{PT2.7}
 F_\delta (t)\leq F(t)=(J_\beta G(t))^\lambda \leq\sqrt{\frac{a^{\lambda}_{2}}{a_1}}(J_\beta G_\delta (t))^\lambda
\end{equation}
which by \eqref{PT2.5} and \eqref{8.1} holds for all other $t$ as well.

Next let $\varphi(x)=e^{-\psi(x)}$ where $\psi(x)=\sqrt{1+|x|^2}-1$.
Then for $\varepsilon\in(0,1)$, $\gamma>1$, and
$|\xi-x|<\gamma\sqrt{2}$ we have
$$\frac{\varphi(\varepsilon\xi)}{\varphi(\varepsilon x)}=e^{-(\psi(\varepsilon\xi)-\psi(\varepsilon x)}\geq e^{-\varepsilon|\xi-x|}\geq e^{-\varepsilon\gamma\sqrt{2}}.$$
Thus defining
$f_\varepsilon ,g_\varepsilon :\mathbb{R}^n
\times\mathbb{R}\to[0,\infty)$ by
$$f_\varepsilon (x,t)=\varphi(\varepsilon x)a_1 F_\delta
(t)\quad\text{ and }\quad g_\varepsilon (x,t)=\varphi(\varepsilon
x)^\sigma a_2 G_\delta (t)$$
we find for $|\xi-x|<\gamma\sqrt{2}$ and $\tau\in\mathbb{R}$ that
\[f_\varepsilon (\xi,\tau)\geq\varphi(\varepsilon x)e^{-\varepsilon
  \gamma\sqrt{2}}a_1 F_\delta (\tau)\quad\text{ and }\quad g_\varepsilon
(\xi,\tau)\geq\varphi(\varepsilon x)^\sigma
e^{-\sigma\varepsilon\gamma\sqrt{2}}a_2 G_\delta (\tau).\]
Hence for $(x,t)\in\mathbb{R}^n \times(0,2)$ we have
\begin{equation}\label{PT2.8}
 J_\alpha f_\varepsilon (x,t)\geq\varphi(\varepsilon x)e^{\varepsilon\gamma\sqrt{2}}a_1 \int^{t}_{0}\frac{(t-\tau)^{\alpha-1}}{\Gamma(\alpha)}F_\delta (\tau)\int_{|\xi-x|<\gamma\sqrt{2}}\Phi_1 (x-\xi,t-\tau)\,d\xi \,d\tau
\end{equation}
and
\begin{equation}\label{PT2.9}
 J_\beta g_\varepsilon (x,t)\geq\varphi(\varepsilon x)^\sigma e^{-\sigma\varepsilon\gamma\sqrt{2}}a_2 \int^{t}_{0}\frac{(t-\tau)^{\beta-1}}{\Gamma(\beta)}G_\delta (\tau)\int_{|\xi-x|<\gamma\sqrt{2}}\Phi_1 (x-\xi,t-\tau)\,d\xi \,d\tau.
\end{equation}
But for $x,\xi\in\mathbb{R}^n$ and $0<\tau<t<2$ we find making the
change of variables $z=\frac{x-\xi}{\sqrt{4(t-\tau)}}$ that
\begin{align}\label{PT2.10}
 \notag \int_{|\xi-x|<\gamma\sqrt{2}}\Phi_1 (x-\xi,t-\tau)\,d\xi&\geq\int_{|\xi-x|<\gamma\sqrt{t-\tau}}\frac{1}{(4\pi(t-\tau))^{n/2}}e^{-\frac{|x-\xi|^2}{4(t-\tau)}}d\xi\\
 &=\frac{1}{\pi^{n/2}}\int_{|z|<\gamma/2}e^{-|z|^2}dz=:I(\gamma)\to 1
\end{align}
as $\gamma\to\infty$.  Thus by \eqref{PT2.8}, \eqref{PT2.9}, \eqref{PT2.6} and
\eqref{PT2.7} we have for $(x,t)\in\mathbb{R}^n \times(0,1+\delta)$
that
\begin{align}\label{PT2.11}
 \notag \frac{(J_\alpha f_\varepsilon (x,t))^\sigma}{g_\varepsilon (x,t)}&\geq\frac{\varphi(\varepsilon x)^\sigma e^{-\sigma\varepsilon\gamma\sqrt{2}}a^{\sigma}_{1}I(\gamma)^\sigma (J_\alpha F_\delta (t))^\sigma}{\varphi(\varepsilon x)^\sigma a_2 G_\delta (t)}\\
 &\geq\sqrt{\frac{a^{\sigma}_{1}}{a_2}}I(\gamma)^\sigma e^{-\sigma\varepsilon\gamma\sqrt{2}}
\end{align}
and
\begin{align}\label{PT2.12}
 \notag \frac{(J_\beta g_\varepsilon (x,t))^\lambda}{f_\varepsilon (x,t)}&\geq\frac{\varphi(\varepsilon x)^{\lambda\sigma} e^{-\lambda\sigma\varepsilon\gamma\sqrt{2}}a^{\lambda}_{2}I(\gamma)^\lambda (J_\beta G_\delta (t))^\lambda}{\varphi(\varepsilon x) a_1 F_\delta (t)}\\
 &\geq\sqrt{\frac{a^{\lambda}_{2}}{a_1}}I(\gamma)^\lambda e^{-\lambda\sigma\varepsilon\gamma\sqrt{2}}
\end{align}
by \eqref{1.6.3}.

So first choosing $\gamma$ so large that
$$\sqrt{\frac{a^{\sigma}_{1}}{a_2}}I(\gamma)^\sigma >1\quad\text{ and }\quad\sqrt{\frac{a^{\lambda}_{2}}{a_1}}I(\gamma)^\lambda >1$$
(we can do this by \eqref{PT2.3} and \eqref{PT2.10}) and then choosing
$\varepsilon>0$ so small that \eqref{PT2.11} and \eqref{PT2.12} are
both greater than one we see that $f:=f_\varepsilon$ and
$g=g_\varepsilon$ satisfy \eqref{1.2} in
$\mathbb{R}^n \times(0,1+\delta)$.  Hence, since $F_\delta$ and
$G_\delta$, and thus $f(x,t)$ and $g(x,t)$, are identically zero in
$\mathbb{R}^n \times((-\infty,0]\cup[1+\delta, \infty))$ we have that
$f$ and $g$ satisfy \eqref{1.2}, \eqref{1.3}.

From the exponential decay of $\varphi(x)$ as $|x|\to\infty$, we find
that $f$ and $g$ satisfy \eqref{1.16}.  Also since $f$ and $g$ are
uniformly continuous and bounded on $\mathbb{R}^n \times\mathbb{R}$
and
$$\int^{b}_{a}\int_{\mathbb{R}^n}\Phi_\alpha
(x,t)\,dx\,dt=\frac{1}{\Gamma(\alpha+1)}(b^\alpha -a^\alpha )\quad \text{ for }a<b$$
we easily check that \eqref{1.17} holds.

Finally, from \eqref{PT2.4} we see for $0<t<1$ that
$$f(0,t)=a_1 F(t)\geq\left(\frac{N_1}{M_1}\right)^{\frac{\lambda\sigma}{1-\lambda\sigma}}M^{\frac{\lambda\sigma}{1-\lambda\sigma}}_{1}t^{\gamma_1}=N^{\frac{\lambda\sigma}{1-\lambda\sigma}}_{1}t^{\gamma_1}$$
and
$$g(0,t)=a_2 G(t)\geq\left(\frac{N_2}{M_2}\right)^{\frac{\lambda\sigma}{1-\lambda\sigma}}M^{\frac{\lambda\sigma}{1-\lambda\sigma}}_{2}t^{\gamma_2}=N^{\frac{\lambda\sigma}{1-\lambda\sigma}}_{2}t^{\gamma_2}$$
and consequently we obtain \eqref{1.18} and \eqref{1.19}.  Hence
\eqref{1.20} and \eqref{1.21} follow from \eqref{1.2}, \eqref{1.3}.
\end{proof}

\begin{proof}[Proof of Theorem \ref{thm3.4}]
  By Remark \ref{rem4.2} with $T=1$ we can assume $K_1 =K_2 =1$.  Define
  $\bar{f},\bar{g}:\mathbb{R}^n \times\mathbb{R}\to[0,\infty)$ by
\begin{equation}\label{PT3.1}
 \bar{f}(x,t)=F(t)\raisebox{2pt}{$\chi$}_\Omega(x,t)\quad\text{ and }\quad\bar{g}(x,t)=G(t)\raisebox{2pt}{$\chi$}_\Omega(x,t)
\end{equation}
where $F$ and $G$ are defined in Remark \ref{rem4.1} and $\Omega=\{|x|^2 <t\}$.  Then using
Lemma \ref{lem7.4} and the fact that $\alpha+\gamma_1 =\gamma_2 /\sigma$ we
obtain for $|x|^2 <t$ that
\begin{align*}
 J_\alpha \bar{f}(x,t)&=\int^{t}_{0}\frac{(t-\tau)^{\alpha-1}}{\Gamma(\alpha)}\left(\int_{|\xi|^2 <\tau}\Phi_1 (x-\xi,t-\tau)\,d\xi\right)F(\tau)\,d\tau\\
 &\geq C\int^{3t/4}_{t/4}(t-\tau)^{\alpha-1}\tau^{\gamma_1}\,d\tau\\
 &=Ct^{\alpha+\gamma_1}=Ct^{\gamma_2 /\sigma}\\
 &=CG(t)^{1/\sigma}=C\bar{g}(x,t)^{1/\sigma}
\end{align*}
which also holds in $\mathbb{R}^n \times\mathbb{R}\backslash\Omega$ because $\bar{g}=0$ there.  Similarly 
$$J_\beta \bar{g}(x,t)\geq C\bar{f}(x,t)^{1/\lambda}\quad\text{ in }\mathbb{R}^n \times\mathbb{R}.$$
Thus letting
\begin{equation}\label{PT3.2}
 f=L_1\bar{f}\quad\text{ and }\quad g=L_2 \bar{g}
\end{equation}
for appropriately chosen positive constants $L_1$ and $L_2$, the function $f$ and $g$ will satisfy \eqref{1.1}--\eqref{1.3}.

It follows from \eqref{PT3.1}, \eqref{PT3.2} and the definition of $F$
and $G$ in Remark \ref{rem4.1} that there exists $N>0$ such that \eqref{1.22}
holds in $\Omega$.  Thus, since $f$ and $g$ solve \eqref{1.2} we obtain
\eqref{1.23}, provided we decrease $N$ if necessary. 
\end{proof}

\begin{rem}\label{rem5.1}
  Suppose \eqref{1.4}, \eqref{1.5}, \eqref{1.6.1}, and \eqref{T6.1}
  hold.  We will need for the proof of Theorems \ref{thm3.5} and
  \ref{thm3.6} some observations concerning the graphs of the straight
  lines in the $\xi\eta$-plane given by
\begin{equation}\label{4.23}
 \xi=(\eta-\beta)\lambda \quad\text{ and }\quad\eta=(\xi-\alpha)\sigma.
\end{equation}
These lines intersect the vertical line $\xi=\frac{n+2}{2p}$ at 
\[P_2 =\left(\frac{n+2}{2p},\eta_2 \right) \quad\text{and}\quad P_3 =\left(\frac{n+2}{2p},\eta_3 \right),\] 
respectively, where
\begin{equation}\label{4.22}
 \eta_2 =\beta+\frac{n+2}{2p\lambda}\quad\text{ and }\quad\eta_3 =\left(\frac{n+2}{2p}-\alpha\right)\sigma=\frac{(n+2)\sigma}{2q\sigma_0}
\end{equation}
by \eqref{2.12}.  Thus
$$\eta_3 <(=,>)\frac{n+2}{2q}\quad\text{ if }\sigma<(=,>)\sigma_0 .$$
Moreover, it follows from \eqref{4.22} and \eqref{T6.1} that
\begin{equation}\label{4.25}
 \eta_2 =\frac{\mu(\lambda)}{\sigma}\eta_3 <\eta_3
\end{equation}
and 
it follows from \eqref{1.5} and \eqref{T6.1}
(see Figure \ref{fig1}) that $\lambda>\lambda_0$ where $\lambda_0$ is
defined in \eqref{2.12}.  Thus by \eqref{4.22}
\begin{equation}\label{4.24}
 \eta_2 <\beta+\frac{n+2}{2p\lambda_0}=\frac{n+2}{2q}.
\end{equation}
The lines \eqref{4.23} are graphed in Figures \ref{fig2}a,
\ref{fig2}b, and \ref{fig2}c when
$\sigma<\sigma_0 ,\sigma=\sigma_0$, and $\sigma>\sigma_0$
respectively.
\end{rem}

\begin{figure}
\begin{tikzpicture}[scale=0.55]
\draw [fill=lime, lime] (4,2) -- (4,3) -- (2.667,1.667);
\draw [<->] [thick] (0,4.5) -- (0,0) -- (4.5,0);
\draw (0,0) rectangle (4,4);
\draw (1,0) -- (4,3);
\draw (0,1) -- (4,2);
\node [right] at (4.5,0) {$\xi$};
\node [above] at (0,4.5) {$\eta$};
\node [left] at (0,4) {$\frac{n+2}{2q}$};
\node [below] at (4,0) {$\frac{n+2}{2p}$};
\node [left] at (0,1) {$\beta$};
\node [below] at (1,0) {$\alpha$};
\node [right] at (4,2) {$P_2$};
\node [right] at (4,3) {$P_3=P_0$};
\node [below] at (2.77,1.68) {$P_4$};
\node [right] at (-1.0,-2.0) {Figure \ref{fig2}a: $\sigma<\sigma_0$};

\draw [fill=lime, lime] (12.7,2) -- (12.7,4) -- (10.8538,1.5384);
\draw [<->] [thick] (8.7,4.5) -- (8.7,0) -- (13.2,0);
\draw (8.7,0) rectangle (12.7,4);
\draw (9.7,0) -- (12.7,4);
\draw (8.7,1) -- (12.7,2);
\node [right] at (13.2,0) {$\xi$};
\node [above] at (8.7,4.5) {$\eta$};
\node [left] at (8.7,4) {$\frac{n+2}{2q}$};
\node [below] at (12.7,0) {$\frac{n+2}{2p}$};
\node [left] at (8.7,1) {$\beta$};
\node [below] at (9.7,0) {$\alpha$};
\node [right] at (12.7,2) {$P_2$};
\node [right] at (12.7,4) {$P_3=P_0$};
\node [below] at (11,1.60) {$P_4$};
\node [right] at (7.7,-2.0) {Figure \ref{fig2}b: $\sigma=\sigma_0$};

\draw [fill=lime, lime] (21.5,2) -- (21.5,4) -- (20.9,4) -- (19.38235,1.47059);
\draw [<->] [thick] (17.5,4.5) -- (17.5,0) -- (22.0,0);
\draw (17.5,0) rectangle (21.5,4);
\draw (18.5,0) -- (21.5,5);
\draw (17.5,1) -- (21.5,2);
\draw (21.5,4) -- (21.5,5);
\node [right] at (22.0,0) {$\xi$};
\node [above] at (17.5,4.5) {$\eta$};
\node [left] at (17.5,4) {$\frac{n+2}{2q}$};
\node [below] at (21.5,0) {$\frac{n+2}{2p}$};
\node [left] at (17.5,1) {$\beta$};
\node [below] at (18.5,0) {$\alpha$};
\node [right] at (21.5,2) {$P_2$};
\node [right] at (21.5,4) {$P_0$};
\node [right] at (21.5,5) {$P_3$};
\node [below] at (19.6,1.57) {$P_4$};
\node [right] at (16.5,-2.0) {Figure \ref{fig2}c: $\sigma>\sigma_0$};

\end{tikzpicture}
\caption{Graphs of lines \eqref{4.23}.}
\label{fig2}
\end{figure}

\begin{proof}[Proof of Theorem \ref{thm3.5}]
  Since $|R_j |<\infty$ for $j=1,2,\dots$, to prove Theorem \ref{thm3.5} it
  suffices to show for each $\varepsilon>0$ there exist
\begin{equation}\label{T6.2}
 r\in(p,p+\varepsilon)\quad\text{ and }\quad s\in(s_0 ,s_0 +\varepsilon)
\end{equation}
such that the conclusion of Theorem \ref{thm3.5} holds.

We will use the notation and observations in Remark \ref{rem5.1}.  Let $P_0 =(\xi_0 ,\eta_0 )$ where 
$$\xi_0 =\frac{n+2}{2p}\quad\text{ and }\quad\eta_0 =
\begin{cases}
 \eta_3 & \text{if }\sigma<\sigma_0 \\
 \frac{n+2}{2q} & \text{if }\sigma\geq\sigma_0 .
\end{cases}
$$
Then
\[
\eta_0=\frac{n+2}{2s_0}\quad \text{and}\quad
P_0 =
\begin{cases}
 P_3 & \text{if }\sigma<\sigma_0 \\
 (\frac{n+2}{2p},\frac{n+2}{2q}) & \text{if }\sigma\geq\sigma_0 .
\end{cases}
\]
The point $P_0$ is graphed in Figure  \ref{fig2}. It follows from
Figure \ref{fig2} that there exist points $P_1(\xi_1,\eta_1)$ in the
open shaded region arbitrarily close to $P_0$. More precisely, fixing
$\varepsilon >0$,  there exist $\xi_1\in(0,\xi_0)$ and
$\eta_1\in(0,\eta_0)$ such that 
\begin{equation}\label{8.14}
 \xi_1 <(\eta_1 -\beta)\lambda \quad\text{and}\quad
 \eta_1 <(\xi_1 -\alpha)\sigma
\end{equation}
and
\begin{equation}\label{8.13}
\frac{n+2}{2(p+\varepsilon)}<\xi_1<\xi_0=\frac{n+2}{2p}\quad\text{and}\quad 
\frac{n+2}{2(s_0+\varepsilon)}<\eta_1<\eta_0=\frac{n+2}{2s_0}\le\frac{n+2}{2q}. 
\end{equation}
Thus defining $r$ and $s$ by
\begin{equation}\label{8.13.5}
\frac{n+2}{2r}=\xi_1\quad\text{and}\quad\frac{n+2}{2s}=\eta_1
\end{equation}
we have $r$ and $s$ satisfy \eqref{T6.2}.

Define $f_0 ,g_0 :\mathbb{R}^n \times\mathbb{R}\to\mathbb{R}$ by
\begin{equation}\label{8.15}
 f_0(x,t)=\left(\frac{1}{t}\right)^{\xi_1}\raisebox{2pt}{$\chi$}_{\Omega_0}(x,t)\quad\text{
   and }\quad g_0(x,t)=\left(\frac{1}{t}\right)^{\eta_1}\raisebox{2pt}{$\chi$}_{\Omega_0}(x,t)
\end{equation}
where
$$\Omega_0 =\{(x,t)\in\mathbb{R}^n \times\mathbb{R}:|x|^2 <t<1\}.$$
Then by \eqref{8.13} and Lemma \ref{lem7.6} we have
\begin{equation}\label{8.16}
 f_0 \in L^p (\mathbb{R}^n \times\mathbb{R})\quad\text{ and }\quad g_0 \in L^q (\mathbb{R}^n \times\mathbb{R})
\end{equation}
and for $(x,t)\in\Omega_0$
\begin{equation}\label{8.17}
 J_\alpha f_0 (x,t)\geq C\left(\frac{1}{t}\right)^{\xi_1 -\alpha}\quad\text{ and
 }\quad J_\beta g_0 (x,t)\geq C\left(\frac{1}{t}\right)^{\eta_1 -\beta}
\end{equation}
where $C=(n,\lambda,\sigma,\alpha,\beta,p,q,r,s)>0$.

Let $\{T_j \}\subset(0,1/2)$ be a sequence such that
\[T_{j+1}<T_j /4,\quad j=1,2,\dots\]
and define
\begin{equation}\label{8.18}
 t_j =T_j /2.
\end{equation}
Then
\begin{equation}\label{8.19}
 \Omega_j :=\{(y,s)\in\mathbb{R}^n \times\mathbb{R}:|y|<\sqrt{T_j -s}\text{ and }t_j <s<T_j \}\subset R_j \subset\Omega_0 .
\end{equation}
Defining $f_j ,g_j :\mathbb{R}^n \times\mathbb{R}\to\mathbb{R}$ by
\begin{equation}\label{8.19.5}
 f_j (x,t)=\left(\frac{1}{T_j
     -t}\right)^{\xi_1}\raisebox{2pt}{$\chi$}_{\Omega_j}(x,t)\quad\text{
   and }\quad g_j (x,t)=\left(\frac{1}{T_j -t}\right)^{\eta_1}\raisebox{2pt}{$\chi$}_{\Omega_j}(x,t)
\end{equation}
we obtain from \eqref{8.13}, \eqref{8.13.5}, \eqref{8.19}, and Lemma
\ref{lem7.7} that
\begin{equation}\label{8.20}
\begin{aligned}
 \| f_j \|^{p}_{L^p (\mathbb{R}^n \times\mathbb{R})}
&=C(n)\int^{T_j -t_j}_{0}\zeta^{(\frac{n+2}{2p}-\xi_1
  )p-1}d\zeta\to0\quad\text{ as }j\to\infty,\\
\| g_j\|^{q}_{L^q (\mathbb{R}^n\times\mathbb{R})}
&=C(n)\int^{T_j-t_j}_0\zeta^{(\frac{n+2}{2q}-\eta_1
                         )q-1}d\zeta\to0\quad \text{ as }j\to\infty,
\end{aligned}
\end{equation}
\begin{equation}\label{8.21}
\|f_j\|_{L^r (R_j )}=\|g_j\|_{L^s (R_j )}=\infty\quad\text{ for }j=1,2,\dots
\end{equation}
and for $(x,t)\in\Omega^{+}_{j}:=\{(x,t)\in\Omega_j :\frac{3T_j}{4}<t<T_j \}$ that
\begin{equation}\label{8.22}
 J_\alpha f_j (x,t)\geq C\left(\frac{1}{T_j -t}\right)^{\xi_1
   -\alpha}\quad\text{ and }\quad J_\beta g_j (x,t)\geq C\left(\frac{1}{T_j -t}\right)^{\eta_1 -\beta}.
\end{equation}

It follows from \eqref{8.15} and \eqref{8.17} that for $(x,t)\in\Omega_0$ we have
$$\frac{f_0 (x,t)}{(J_\beta g_0 (x,t))^\lambda}\leq Ct^{(\eta_1 -\beta)\lambda-\xi_1 >0}$$
and
$$\frac{g_0 (x,t)}{(J_\alpha f_0 (x,t))^\sigma}\leq Ct^{(\xi_1 -\alpha)\sigma-\eta_1 >0}.$$
Thus by \eqref{8.14} and \eqref{8.19} we find that
\begin{equation}\label{8.23}
 \sup_{\Omega_0}\frac{f_0}{(J_\beta g_0 )^\lambda}\leq C,\quad\sup_{\Omega_0}\frac{g_0}{(J_\alpha f_0 )^\sigma}\leq C
\end{equation}
and
\begin{equation}\label{8.24}
 \sup_{\Omega_j}\frac{f_0}{(J_\beta g_0 )^\lambda}\leq 1,\quad\sup_{\Omega_j}\frac{g_0}{(J_\alpha f_0 )^\sigma}\leq1
\end{equation}
by taking a subsequence.

Using \eqref{8.19.5}, \eqref{8.22}, \eqref{8.14} and taking a subsequence we obtain
\begin{align}\label{8.25.1}
 \notag \sup_{\Omega^{+}_{j}}\frac{f_j}{(J_\beta g_j )^\lambda}&\leq C\sup_{(x,t)\in\Omega^{+}_{j}}(T_j -t)^{(\eta_1 -\beta)\lambda-\xi_1}\\
 &\leq C(T_j -t_j )^{(\eta_1 -\beta)\lambda-\xi_1}<1
\end{align}
and similarly
\begin{equation}\label{8.25.2}
 \sup_{\Omega^{+}_{j}}\frac{g_j}{(J_\alpha f_j )^{\sigma}}\leq C(T_j -t_j )^{(\xi_1 -\alpha)\sigma-\eta_1}<1.
\end{equation}

It follows from \eqref{8.15}, \eqref{8.19.5},\eqref{8.19}, and \eqref{8.18} that
\begin{equation}\label{8.26.1}
 \sup_{\Omega_j}\frac{f_0}{f_j}=\sup_{(x,t)\in\Omega_j}\frac{(T_j -t)^{\xi_1}}{t^{\xi_1}}\leq1
\end{equation}
and
\begin{equation}\label{8.26.2}
 \sup_{\Omega_j}\frac{g_0}{g_j}=\sup_{(x,t)\in\Omega_j}\frac{(T_j -t)^{\eta_1}}{t^{\eta_1}}\leq1
\end{equation}
and letting $\Omega^{-}_{j}=\Omega_j \backslash\Omega^{+}_{j}$, using \eqref{8.17}, \eqref{8.19.5}, \eqref{8.19}, \eqref{8.18} and \eqref{8.14} and taking a subsequence we obtain
\begin{align}\label{8.27.1}
 \notag \sup_{\Omega^{-}_{j}}\frac{f_j}{(J_\beta g_0 )^\lambda}&\leq C\sup_{(x,t)\in\Omega^{-}_{j}}\frac{t^{(\eta_1-\beta)\lambda}}{(T_j -t)^{\xi_1}}\leq C\frac{(2t_j )^{(\eta_1 -\beta)\lambda}}{(t_j /2)^{\xi_1}}\\
 &=Ct^{(\eta_1 -\beta)\lambda-\xi_1}_{j}<\frac{1}{2}
\end{align}
and similarly
\begin{equation}\label{8.27.2}
 \sup_{\Omega^{-}_{j}}\frac{g_j}{(J_\alpha f_0 )^\lambda}\leq Ct^{(\xi_1 -\alpha)\sigma-\eta_1}_{j}<\frac{1}{2}.
\end{equation}

Taking an appropriate subsequence of $(f_j ,g_j )$ and letting
\[f=f_0 +\sum^{\infty}_{j=1}f_j \quad\text{ and }\quad g=g_0 +\sum^{\infty}_{j=1}g_j\]
we see from \eqref{8.16} and \eqref{8.20} that \eqref{T6.3} holds.  In $\Omega^{+}_{j}$ we have by \eqref{8.24} and \eqref{8.25.1} that
\begin{align*}
 f&=f_0 +f_j \leq(J_\beta g_0 )^\lambda +(J_\beta g_j )^\lambda \\
 &\leq C(\lambda)(J_\beta (g_0 +g_j ))^\lambda \leq C(\lambda)(J_\beta g)^\lambda
\end{align*}
and similarly by \eqref{8.24} and \eqref{8.25.2} that $g\leq C(\sigma)(J_\alpha f)^\sigma$.

In $\Omega^{-}_{j}$ we have by \eqref{8.26.1} and \eqref{8.27.1} that
$$f=f_0 +f_j \leq2f_j \leq(J_\beta g_0 )^\lambda \leq(J_\beta g)^\lambda$$
and similarly by \eqref{8.26.2} and \eqref{8.27.2} that $g\leq(J_\alpha f)^\sigma$.  In $\Omega_0 \backslash\bigcup^{\infty}_{j=1}\Omega_j$ we have by \eqref{8.23} that
$$f=f_0 \leq C(J_\beta g_0 )^\lambda \leq C(J_\beta g)^\lambda$$
and
$$g=g_0 \leq C(J_\alpha f_0 )^\sigma \leq (CJ_\alpha f)^\sigma .$$
In $\mathbb{R}^n \times\mathbb{R}\backslash\Omega_0$, $f=0\leq(J_\beta g)^\lambda$ and $g=0\leq(J_\alpha f)^\sigma$.  Thus, after scaling $f$ and $g$ we see that $f$ and $g$ are solutions of \eqref{1.2} and \eqref{1.3}.  Also \eqref{T6.4} holds by \eqref{8.21}. 
\end{proof}

\begin{proof}[Proof of Theorem \ref{thm3.6}]
  We will use the notation and observations in Remark \ref{rem5.1}.
Let $P_4 =(\xi_4 ,\eta_4 )$ be the point where the lines \eqref{4.23}
intersect.  It follows from \eqref{4.25} and \eqref{4.24} (see Figure \ref{fig2}) that
\begin{equation}\label{8.28}
 0<\xi_4 <\frac{n+2}{2p}\quad\text{ and }\quad 0<\eta_4 <\frac{n+2}{2q}.
\end{equation}
Thus, since solving the system \eqref{4.23} yields
\begin{equation}\label{8.28.5}
 \xi_4 =\frac{n+2}{2r_0}\quad\text{ and }\quad\eta_4 =\frac{n+2}{2s_0},
\end{equation}
we see that $r_0 >p$ and $s_0 >q$.

Define $f_0 ,g_0 :\mathbb{R}^n \times\mathbb{R}\to\mathbb{R}$ by
\[f_0(x,t)=\left(\frac{1}{t}\right)^{\xi_4}\raisebox{2pt}{$\chi$}_{\Omega_0}(x,t)\quad\text{
  and }\quad g_0 (x,t)=\left(\frac{1}{t}\right)^{\eta_4}\raisebox{2pt}{$\chi$}_{\Omega_0}(x,t)\]
where $\Omega_0$ is as in Lemma \ref{lem7.6}.  Then by \eqref{8.28},
Lemma \ref{lem7.6} and the fact that $P_4$ satisfies \eqref{4.23} we
have
\begin{equation}\label{8.29}
 f_0 \in X^p, \quad g_0 \in X^q
\end{equation}
and
\begin{equation}\label{8.30}
 f_0 \leq C(J_\beta g_0 )^\lambda \quad\text{ and }\quad g_0 \leq C(J_\alpha f_0 )^\sigma \quad\text{ in }\mathbb{R}^n \times\mathbb{R}
\end{equation}
where in this proof $C$ is a positive constant depending on $n,
\lambda, \sigma, \alpha, \beta, p, q$, whose values may change from line to line.

Let $\{T_j \},\{t_j \}\subset(2,\infty)$ satisfy $T_{j+1}>4T_j$ and $T_j =2t_j$ and define $f_j ,g_j :\mathbb{R}^n \times\mathbb{R}\to\mathbb{R}$ by
\begin{equation}\label{8.31}
 f_j (x,t)=\left(\frac{1}{T_j
     -t}\right)^{\xi_4}\raisebox{2pt}{$\chi$}_{\Omega_j}(x,t)\quad\text{
   and }\quad g_j (x,t)=\left(\frac{1}{T_j -t}\right)^{\eta_4}\raisebox{2pt}{$\chi$}_{\Omega_j}(x,t)
\end{equation}
where
$$\Omega_j =\{(x,t)\in\mathbb{R}^n \times(t_j ,T_j ):|x|<\sqrt{T_j -t}\}.$$
Then
\begin{equation}\label{8.32}
 \Omega_j \subset R_j \subset\Omega_0 ,\quad\Omega_j \cap\Omega_k =\emptyset\quad\text{ for }j\neq k
\end{equation}
\begin{equation}\label{8.33}
 \inf\{t:(x,t)\in\Omega_j \}=t_j \to\infty\quad\text{ as }j\to\infty,
\end{equation}
and by \eqref{8.31}, \eqref{8.28}, Lemma \ref{lem7.7}, and the fact that $P_4$ satisfies \eqref{4.23} we have
\begin{equation}\label{8.34}
 f_j \in L^p (\mathbb{R}^n \times\mathbb{R}),\quad g_j \in L^q (\mathbb{R}^n \times\mathbb{R})
\end{equation}
and
\[f_j \leq C(J_\beta g_j )^\lambda \quad\text{ and }\quad g_j \leq C(J_\alpha f_j )^\sigma \quad\text{ in }\Omega^{+}_{j}\]
where
$$\Omega^{+}_{j}=\{(x,t)\in\Omega_j :\frac{3T_j}{4}<t<T_j \}.$$
It follows therefore from \eqref{8.30} that
\begin{equation}\label{8.35}
 f_0 +f_j \leq C((J_\beta g_0 )^\lambda +(J_\beta g_j )^\lambda )\leq C(J_\beta (g_0 +g_j ))^\lambda \quad\text{ in }\Omega^{+}_{j}
\end{equation}
and similarly
\begin{equation}\label{8.35.1}
 g_0 +g_j \leq C(J_\alpha (f_0 +f_j ))^\sigma \quad\text{ in }\Omega^{+}_{j}.
\end{equation}
In $\Omega^{-}_{j}=\Omega_j \backslash\Omega^{+}_{j}$ we have
$$\frac{f_j}{f_0}=\left(\frac{t}{T_j -t}\right)^{\xi_4}\leq\left(\frac{3T_j /4}{T_j /4}\right)^{\xi_4}=3^{\xi_4}$$
and similarly
$$\frac{g_j}{g_0}\leq3^{\eta_4}.$$
Thus we obtain from \eqref{8.30} that
\begin{equation}\label{8.36}
  f_0 +f_j \leq Cf_0 \leq C(J_\beta g_0 )^\lambda 
\leq C(J_\beta (g_0 +g_j ))^\lambda \quad\text{ in }\Omega^{-}_{j}
\end{equation}
and similarly that
\begin{equation}\label{8.36.1}
 g_0 +g_j \leq C(J_\alpha (f_0 +f_j ))^\sigma \quad\text{ in }\Omega^{-}_{j}.
\end{equation}
Let
\[f=f_0 +\sum^{\infty}_{j=1}f_j \quad\text{ and }\quad g=g_0 +\sum^{\infty}_{j=1}g_j .\]
Then clearly $f$ and $g$ satisfy \eqref{1.3} and by \eqref{8.29}, \eqref{8.34}, and \eqref{8.33} we see that $f$ and $g$ satisfy \eqref{T7.2}.

In $\Omega_j$ we have by \eqref{8.32}$_2$, \eqref{8.35}, \eqref{8.35.1}, \eqref{8.36}, and \eqref{8.36.1} that
\begin{align*}
 &f=f_0 +f_j \leq C(J_\beta (g_0 +g_j ))^\lambda \leq C(J_\beta g)^\lambda \\
 &g=g_0 +g_j \leq C(J_\alpha (f_0+f_j ))^\sigma \leq C(J_\alpha f)^\sigma
\end{align*}
and in $(\mathbb{R}^n \times\mathbb{R})\backslash\cup^{\infty}_{j=1}\Omega_j$ we have by \eqref{8.30} that
$$f=f_0 \leq C(J_\beta g_0 )^\lambda \leq C(J_\beta g)^\lambda$$
and
$$g=g_0 \leq C(J_\alpha f_0 )^\sigma \leq C(J_\alpha f)^\sigma .$$
Thus after scaling $f$ and $g$, we find that $f$ and $g$ satisfy \eqref{1.2}.

From \eqref{8.28.5}, \eqref{8.31}, \eqref{8.32}$_1$ and Lemma
\ref{lem7.7} we find
\[\| f\|_{L^{r_0}(R_j )}\geq\| f_j \|_{L^{r_0}(R_j )}=\infty \quad\text{ for }j=1,2,\dots\]
and
\[\| g\|_{L^{s_0}(R_j )}\geq\| g_j \|_{L^{s_0}(R_j )}=\infty\quad\text{ for }j=1,2,\dots.\]
Thus, since $|R_j |<\infty$, we have \eqref{T6.4} holds for all $r$
and $s$ satisfying \eqref{T7.1}. 
\end{proof}

\appendix

\section{Auxiliary lemmas}\label{secA}
In this appendix we provide some lemmas needed for the proofs of our
results in Section \ref{sec3} dealing with solutions of the $J_\alpha$
problem \eqref{1.1}--\eqref{1.4}. See \cite[Section 7]{T} for the proofs of these lemmas.

Let $\Omega=\mathbb{R}^n \times(a,b)$ where $n\geq1$ and $a<b$.  The following two lemmas give estimates for the convolution
\begin{equation}\label{7.1}
 (V_{\alpha,\Omega}f)(x,t)=\iint_{\Omega}\Phi_\alpha (x-\xi,t-\tau)f(\xi,\tau)\, d\xi \,d\tau
\end{equation}
where $\alpha>0$ and $\Phi_\alpha$ is defined in \eqref{I.9}.

\begin{rem}\label{rem7.1}
  Note that if $f:\mathbb{R}^n \times\mathbb{R}\to\mathbb{R}$ is a
  nonnegative measurable function such that
  $\| f\|_{L^\infty (\mathbb{R}^n \times\mathbb{R}_a)}=0$ then
 $$V_{\alpha,\Omega}f=J_\alpha f\quad\text{in } \Omega:=\mathbb{R}^n \times(a,b).$$
\end{rem}

\begin{lem}\label{lem7.1}
 For $\alpha>0,\,\Omega=\mathbb{R}^n \times(a,b)$ and $f\in L^\infty (\Omega)$ we have
 $$\| V_{\alpha,\Omega}f\|_{L^\infty (\Omega)}\leq\frac{(b-a)^\alpha}{\Gamma(\alpha+1)}\| f\|_{L^\infty (\Omega)}.$$
\end{lem}

\begin{lem}\label{lem7.2}
 Let $p,q\in[1,\infty]$, $\alpha$, and $\delta$ satisfy
 \begin{equation}\label{7.2}
  0\leq\delta:=\frac{1}{p}-\frac{1}{q}<\frac{2\alpha}{n+2}<1.
 \end{equation}
 Then $V_{\alpha,\Omega}$ maps $L^p (\Omega)$ continuously into $L^q (\Omega)$ and for $f\in L^p (\Omega)$ we have
 $$\| V_{\alpha,\Omega}f\|_{L^q (\Omega)}\leq M\| f\|_{L^p (\Omega)}$$
 where 
 $$M=C(b-a)^{\frac{2\alpha-(n+2)\delta}{2}}\text{ for some constant }C=C(n,\alpha,\delta).$$
\end{lem}

\begin{lem}\label{lem7.4}
 Suppose $x\in\mathbb{R}^n$ and $t,\tau\in(0,\infty)$ satisfy
 \begin{equation}\label{7.8}
  |x|^2 <t \quad\text{and}\quad  \frac{t}{4}<\tau<\frac{3t}{4}.
 \end{equation}
 Then
 $$\int_{|\xi|^2 <\tau}\Phi_1 (x-\xi,t-\tau)\,d\xi\geq C(n)>0$$
 where $\Phi_\alpha$ is defined by \eqref{I.9}.
\end{lem}

\begin{lem}\label{lem7.5}
 For $\tau<t\leq T$ and $|x|\leq\sqrt{T-t}$ we have
 $$\int_{|\xi|<\sqrt{T-\tau}}\Phi_1(x-\xi,t-\tau)\,d\xi\geq C$$
 where $C=C(n)$ is a positive constant. 
\end{lem}

\begin{lem}\label{lem7.6}
 Suppose $\alpha>0$, $\gamma>0$, $p\geq1$, and 
 \[f_0
 (x,t)=\left(\frac{1}{t}\right)^{\frac{n+2}{2p}-\gamma}\raisebox{2pt}{$\chi$}_{\Omega_0}(x,t)\quad\text{
   where }\Omega_0 =\{(x,t)\in \mathbb{R}^n\times\mathbb{R}:|x|^2 <t\}.\]
 Then $f_0 \in X^p$ and
 \[C_1 \left(\frac{1}{t}\right)^{\frac{n+2}{2p}-\gamma-\alpha}\leq J_\alpha f_0 (x,t)\leq C_2 \left(\frac{1}{t}\right)^{\frac{n+2}{2p}-\gamma-\alpha}\quad\text{for }(x,t)\in\Omega_0\]
 where $C_1$ and $C_2$ are positive constants depending only on $n,\alpha,\gamma$, and $p$.
\end{lem}

\begin{lem}\label{lem7.7}
 Suppose $\alpha>0$, $\gamma\in\mathbb{R}$, $0\leq t_0 <T,\,p\in[1,\infty)$, and
 \[
f(x,t)=\left(\frac{1}{T-t}\right)^{\frac{n+2}{2p}-\gamma}\raisebox{2pt}{$\chi$}_\Omega(x,t)
\]
 where
 $$\Omega=\{(x,t)\in\mathbb{R}^n \times(t_0 ,T):|x|<\sqrt{T-t}\}.$$
 Then
 $$J_\alpha f(x,t)\geq C\left(\frac{1}{T-t}\right)^{\frac{n+2}{2p}-\gamma-\alpha}$$
 for $(x,t)\in\Omega^+ :=\{(x,t)\in\Omega:\frac{T+t_0}{2}<t<T\}$ where $C=C(n,\alpha,\gamma,p)>0$.  Moreover,
 \begin{equation}\label{7.13}
  f\in L^p (\mathbb{R}^n \times\mathbb{R})\text{ if and only if }\gamma>0
 \end{equation}
 and in this case
 \begin{equation}\label{7.14}
  \| f\|^{p}_{L^p (\mathbb{R}^n \times\mathbb{R})}=C(n)\int^{T-t_0}_{0}s^{\gamma p-1}ds.
 \end{equation}
\end{lem}

\end{document}